\documentclass[parskip=half]{scrartcl}
\usepackage[utf8]{inputenc}
\usepackage[T1]{fontenc}
\usepackage{lmodern}
\usepackage{amsfonts}
\usepackage{titletoc}
\usepackage[colorlinks=true,linkcolor=black,urlcolor=black,citecolor=black,bookmarks]{hyperref}
\usepackage{amsmath}
\usepackage{amssymb}
\usepackage{mathrsfs}
\usepackage{mathtools}
\usepackage{amsthm}
\usepackage{tabularx}
\usepackage{enumerate}
\usepackage{thmtools}
\usepackage{tikz}
\usepackage{float}
\usepackage{comment}
\usepackage{tikz-cd} 
\usepackage[nameinlink]{cleveref}
\usepackage[onehalfspacing]{setspace}
\usepackage[style=alphabetic, giveninits=true, sorting=nyt,maxnames=4,isbn=false,url=false,doi=false]{biblatex}

\declaretheoremstyle[spaceabove=12pt plus 2pt minus 4pt ,spacebelow=12pt plus 2pt minus 4pt,bodyfont=\normalfont]{scdef}
\declaretheoremstyle[spaceabove=12pt plus 2pt minus 4pt,spacebelow=8pt plus 2pt minus 4pt,bodyfont=\itshape]{scthm} 
\declaretheoremstyle[spaceabove=\topsep,spacebelow=12pt plus 2pt minus 4pt,bodyfont=\normalfont,headfont=\itshape,notefont=\itshape,postheadspace=0.3em,qed=\qedsymbol,headformat=\NAME\NOTE,notebraces=\relax]{scpf}
\declaretheorem[numbered=no,style=scthm,name=Theorem, refname={Theorem,Theorems}, Refname={Theorem,Theorem}] {theorem*}
\declaretheorem[style=scthm,numberwithin=section,name=Theorem, refname={Theorem,Theorems}, Refname={Theorem,Theorem}] {theorem} 
\declaretheorem[style=scthm,sharenumber=theorem, name=Lemma, refname={Lemma,Lemmata}, Refname={Lemma,Lemma}] {lemma} 
\declaretheorem[style=scthm,sharenumber=theorem, name=Corollary, refname={Corollary,Corollaries}, Refname={Corollary,Corollaries}] {corollary}
\declaretheorem[style=scdef,sharenumber=theorem, name=Remark, refname={Remark,Remarks}, Refname={Remark,Remarks}] {remark}
\declaretheorem[style=scdef, sharenumber=theorem, name=Definition, refname={Definition,Definitions}, Refname={Definition,Definitions}]{definition}

\declaretheorem[style=scdef, sharenumber=theorem, name=Proposition, refname={Proposition,Propositions}, Refname={Proposition,Propositions}]{proposition}
\declaretheorem[style=scdef, sharenumber=theorem, name=Example, refname={Example,Examples}, Refname={Example,Examples}]{example}

\newcommand{\ad}{\mathrm{ad}}
\newcommand{\Ad}{\mathrm{Ad}}

\newcommand{\Aut}{\mathrm{Aut}}
\newcommand{\Isom}{\mathrm{Isom}}

\newcommand{\mf}{\mathfrak}
\newcommand{\R}{\mathbb{R}}
\newcommand{\C}{\mathbb{C}}
\newcommand{\Z}{\mathbb{Z}}
\newcommand{\N}{\mathbb{N}}
\newcommand{\id}{\mathrm{id}}
\newcommand{\wt}{\widetilde}
\newcommand\extalg{%
  \newlength{\len}%
  \settoheight{\len}{V}%
  \mathbin{%
    \resizebox{0.93\len}{0.93\len}{$\wedge$}%
    \kern-0.1em%
  }}%

  \DeclareMathOperator{\tr}{tr}
  \DeclareMathOperator{\rk}{rk}
  \DeclareMathOperator{\diag}{diag}

\dottedcontents{section}[2.9em]{}{2.9em}{1pc}

\title{Revisiting the Classification of Homogeneous 3-Sasakian and Quaternionic Kähler Manifolds}
\author{Oliver Goertsches, Leon Roschig \& Leander Stecker}

\begin{document}
\maketitle

\begin{abstract}
\begin{center} \textbf{Abstract} \end{center}
We provide a new, self-contained proof of the classification of homogeneous 3-Sasakian manifolds, which was originally obtained by \textsc{Boyer}, \textsc{Galicki} and \textsc{Mann} \cite{BGM}. In doing so, we construct an explicit one-to-one correspondence between simply connected homogeneous 3-Sasakian manifolds and simple complex Lie algebras via the theory of root systems. We also discuss why the real projective spaces are the only non-simply connected homogeneous 3-Sasakian manifolds and derive the famous classification of homogeneous positive quaternionic Kähler manifolds due to \textsc{Wolf} \cite{Wolf} and \textsc{Alekseevskii} \cite{Alek} from our results.
\end{abstract}

\tableofcontents

\section{Introduction}

3-Sasakian geometry is arguably one of the most important odd-dimensional geometries. It provides a rich source of compact Einstein manifolds, lies ``sandwiched'' between the famous hyperkähler, quaternionic Kähler (qK) and Kähler-Einstein geometries and also has several links to algebraic geometry \cite{BG}. \\
The history of the classification of homogeneous manifolds with this geometry is rather long and complicated, as it involves work from as early as 1961 \cite{Boot} and as recently as 2020 \cite{DOP}. In this article, we would like to summarize, revisit and improve upon this classification by proving the following

\begin{theorem} \label{goal}
There is a one-to-one correspondence between simply connected homogeneous $3$-Sasakian manifolds and simple complex Lie algebras.

Given a complex simple Lie algebra $\mathfrak{u}$, choose a maximal root $ \alpha $ of $ \mf{u} $ and let $ \mf{v} $ denote the direct sum of the subspace $ \ker \alpha $ and the root spaces of roots perpendicular to $ \alpha $. Let $ \mathfrak{g} $ and $ \mf{h} $ be the compact real forms of $ \mf{u} $ and $ \mf{v} $, respectively, and write $ \mathfrak{k} \cong \mathfrak{sp}(1) $ for the compact real form of the $ \mf{sl}(2,\C) $-subalgebra defined by $ \alpha $. Let $ B $ denote the Killing form of $ \mf{g} $, set $ \mf{g}_1 = (\mf{h} \oplus \mf{k})^{\perp_B} $ and consider the reductive complement $ \mf{m} = \mf{h}^{\perp_B} = \mf{k} \oplus \mf{g}_1 $. Let $ G $ be the simply connected Lie group with Lie algebra $ \mf{g} $ and let $ H \subset G $ the connected subgroup with Lie algebra $ \mf{h} $. Define a $ G $-invariant Riemannian metric $ g $ on $ M = G/H $ by extending the inner product on $T_{eH}M \cong \mf{m}$ given by
\begin{gather*}
 g|_{\mf{k} \times \mf{k}} = - \frac{1}{4(n+2)} B \, , \qquad g|_{\mf{g}_1 \times \mf{g}_1} = - \frac{1}{8(n+2)} B \, , \qquad g|_{\mf{k} \times \mf{g}_1} = 0 \, .
\end{gather*}
Consider a basis $ X_1,X_2,X_3 $ of $ \mf{k} $ satisfying the commutator relations $ [X_i,X_j] = 2 \varepsilon_{ijk} X_k $ and extend $ X_i \in \mf{m} \cong T_{eH}M $ to a $ G $-invariant vector field $ \xi_i $ on $ M $. Let $\eta_i$ denote the metric dual of $\xi_i$ and $\varphi_i $ the $ G $-invariant endomorphism field defined by extending
\[ \varphi_i |_\mf{k} = \frac{1}{2} \, \ad_{X_i} \, , \qquad \varphi_i|_{\mf{g}_1} = \ad_{X_i} \, . \]
Then, $(g,\xi_i,\eta_i,\varphi_i)_{i=1,2,3}$ is a $G$-invariant $3$-Sasakian structure on $ M $.

Conversely, given a simply connected homogeneous $3$-Sasakian manifold $M$, represented as the quotient $\widetilde{G}/\widetilde{H}$, where $ \widetilde{G} $ is connected and acts effectively, then $\widetilde{G} = \mathrm{Aut}_0(M)$, the connected component of the $3$-Sasakian automorphism group of $M$, and $M$ is the unique space associated with the complexification of the Lie algebra of $ \widetilde{G} $.
\end{theorem}

Using this characterization, we rediscover the list of homogeneous $3$-Sasakian manifolds as given by \textsc{Boyer}, \textsc{Galicki} and \textsc{Mann}:
\begin{corollary} \label{3Slist}
Every homogeneous 3-Sasakian manifold $ M = G/H $ (not necessarily simply connected) is isomorphic to one of the following spaces:
\begin{align*}
\frac{Sp(n+1)}{Sp(n)} \cong S^{4n+3}\, &, \quad \frac{Sp(n+1)}{Sp(n) \times \Z_2} \cong  \R P^{4n+3}\, , \quad \frac{SU(m)}{S(U(m-2) \times U(1))} \, , \\
\frac{SO(k)}{SO(k-4) \times Sp(1)}\, &, \quad \frac{G_2}{Sp(1)} \, , \quad \frac{F_4}{Sp(3)} \, , \quad \frac{E_6}{SU(6)} \,, \quad \frac{E_7}{Spin(12)} \, , \quad \frac{E_8}{E_7} \, .
\end{align*}
To avoid redundancy, we need to assume $ n \geq 0 $, $ m \geq 3 $ and $ k \geq 7 $.
\end{corollary}

As a consequence, we also arrive at the complete list of homogeneous positive qK mani- folds as discovered by \textsc{Wolf} and \textsc{Alekseevskii}:
\begin{corollary} \label{Wolf}
Every homogeneous positive qK manifold is isometric to one of the spaces
\begin{align*}
\frac{Sp(n+1)}{Sp(n)\times Sp(1)} \, &, \quad \frac{SU(m)}{S(U(m-2) \times U(2))} \, , \quad \frac{SO(k)}{SO(k-4) \times SO(4)}\, , \\ & \hspace{-3cm} \frac{G_2}{SO(4)} \, , \quad \frac{F_4}{Sp(3)Sp(1)} \, , \quad \frac{E_6}{SU(6)Sp(1)} \,, \quad \frac{E_7}{Spin(12)Sp(1)} \, , \quad \frac{E_8}{E_7Sp(1)} \, ,
\end{align*}
where the Riemannian metric and quaternionic structure are also determined by \autoref{goal} via the so-called Konishi bundle (see \autoref{fibr} and \Cref{wolf} for details).
\end{corollary}

The discussion of \autoref{goal} and its consequences will be divided into several \mbox{sections:} We begin by recalling basic definitions and features of 3-Sasakian geometry in \mbox{\Cref{3SG}.} We then summarize the history of the classification (\Cref{hist}) and introduce a certain $ \Z $-grading of semisimple complex Lie algebras based on their root systems (\Cref{roots}). The first half of the proof of \autoref{goal} is given in \Cref{const}, where we construct homogeneous 3-Sasakian manifolds from simple Lie algebras. The centerpiece of this article is the converse argument in \Cref{univ}. For reasons that will become apparent during the construction, the special case of the exceptional Lie algebra $ \mf{g}_2 $ needs to be relegated to \Cref{G2}. We complete the proof of \autoref{goal} by showing that no proper subgroup of the identity component $ \Aut_0(M)$ of the automorphism group can act transitively in \Cref{AutM}. In \Cref{Isotropy} we compute the isotropy groups described in \autoref{3Slist} explicitly for the classical spaces and Lie theoretically via Borel-de Siebenthal theory in the exceptional cases. In \Cref{fund}, we show that the only non simply connected homogeneous $3$-Sasakian manifolds are the real projective spaces $\mathbb{R}P^{4n+3}$, which are the $\mathbb{Z}_2$-quotient of the previously described space $S^{4n+3}=Sp(n+1)/Sp(n)$. Finally, since our arguments are independent of the classification of homogeneous positive qK manifolds, they allow for an alternative proof of the latter (\Cref{wolf}). \\
\textbf{Acknowledgements:} The second named author is supported by the German Academic Scholarship Foundation. The third named author thanks Ilka Agricola and Cristina Draper for fruitful discussions about the subject.

\section{Fundamentals of 3-Sasakian Geometry} \label{3SG}

3-Sasakian geometry may be approached from a variety of different starting points, including but not limited to: qK and hyperkähler geometry, Einstein geometry, spin geo\-metry and certain areas of algebraic geometry. For the sake of brevity, we decided to limit the exposition in this article to the necessary minimum. The interested reader is referred to the comprehensive monograph \cite[Chapters 6, 8, 13]{BG}.

\begin{definition}
A \emph{Sasakian structure} is a tuple $ (M^{2n+1},g,\xi,\eta,\varphi) $, where $ \xi $ is a unit length Killing vector field, the one form $ \eta = g(\xi,\cdot) $, the endomorphism field $ \varphi = - \nabla^g \xi $ and the following curvature condition is satisfied:
\[ R(X,\xi)Y = \eta(Y)X - g(X,Y)\xi \quad \forall \; X,Y \in \mf{X}(M) \, . \vspace{-.125cm} \]
We refer to the objects $ \xi $, $ \eta $ and $ \varphi $ as the \emph{Reeb vector field}, \emph{contact form} and \emph{almost complex structure} of the Sasakian structure.
\end{definition}

\begin{remark} \label{Sdef}
The definition entails several important identities. Other common definitions of Sasakian manifolds in the literature are usually some selection of these. For $ X,Y \in \mf{X}(M) $ we have:
\begin{align*}
\varphi^2 = - \id + \eta \otimes \xi \, , \quad g(\varphi X, \varphi Y) = g(X&,Y) - \eta(X)\eta(Y) \, , \quad \varphi \xi = 0 \, , \quad \eta \circ \varphi = 0 \, , \\
g(X,\varphi Y) + g(\varphi X, Y) = 0 \, , \quad d\eta(X,Y) &= 2g(X,\varphi Y) \, , \quad d\eta(\varphi X, \varphi Y) = d\eta(X,Y) \, .
\end{align*}
We also remark that the Reeb vector field $ \xi $ is characterized uniquely by the properties $ \eta(\xi) = 1 $, $ d\eta(\xi, \cdot) = 0 $ and that the cone over a Sasakian manifold admits a Kähler structure.
\end{remark}

\begin{definition}
A \emph{3-Sasakian structure} is a tuple $ (M^{4n+3},g,\xi_i,\eta_i,\varphi_i)_{i=1,2,3} $ such that each $ (M,g,\xi_i,\eta_i,\varphi_i) $ is a Sasakian structure and $ g(\xi_i,\xi_j) = \delta_{ij} $, $ [\xi_i,\xi_j] = 2 \varepsilon_{ijk} \xi_k $, where $ \varepsilon_{ijk} $ denotes the Levi-Civita symbol and $(i,j,k)$ is any permutation of $(1,2,3)$.
\end{definition}

\begin{remark} \label{3Sdef}
For any cyclic permutation $(i,j,k)$ of $ (1,2,3) $ the definition implies the compatibility conditions
\[ \varphi_i \circ \varphi_j - \eta_j \otimes \xi_i = \varphi_k \, , \quad \varphi_i \xi_j = \xi_k \, , \quad \eta_i \circ \varphi_j = \eta_k \, . \]
The cone over a 3-Sasakian manifold admits a hyperkähler structure.
\end{remark}

\begin{proposition}[{\cite[Cor.~13.2.3]{BG}}] \label{Einstein}
Every 3-Sasakian manifold $ M $ of dimension $ 4n+3 $ is Einstein with Einstein constant $ 2(2n+1) $. Moreover, if $ M $ is complete, it is compact with finite fundamental group.
\end{proposition}

\begin{definition} A \emph{3-Sasakian isomorphism} between 3-Sasakian manifolds $ (M,g,\xi_i,\eta_i,\varphi_i) $ and $ (\wt{M},\wt{g},\wt{\xi_i},\wt{\eta_i},\wt{\varphi_i}) $ is an isometry $ \phi: M \to \wt{M} $ which satisfies one of the three equivalent conditions $ \phi_* \xi_i = \wt{\xi_i} $, $ \phi^* \wt{\eta_i} = \eta_i $ or $ \phi_* \circ \varphi_i = \wt{\varphi_i} \circ \phi_* $ for $ i = 1,2,3 $. We will mostly be interested in the case $ (\wt{M},\wt{g},\wt{\xi_i},\wt{\eta_i},\wt{\varphi_i}) = (M,g,\xi_i,\eta_i,\varphi_i) $, in which we call $ \phi $ a \emph{3-Sasakian automorphism of $ M $}. \\
We denote the group of all such transformation by $ \Aut(M) $ and call $ M $ a \emph{homogeneous 3-Sasakian manifold} if $ \Aut(M) $ acts transitively. Homogeneous 3-Sasakian manifolds are in particular Riemannian homogeneous and thus, complete and compact. We remark that the Lie algebra $ \mf{aut}(M) $ consists of the Killing vector fields $ X \in \mf{X}(M) $ such that $ \mathcal{L}_X \xi_i = 0 $, $ \mathcal{L}_X \eta_i = 0 $ and $ \mathcal{L}_X \varphi_i = 0 $ for $ i =1,2,3 $.
\end{definition}

\begin{proposition}[{\cite[Theorem 13.3.13 \& Proposition 13.4.5]{BG}}] \label{fibr} The Reeb vector fields $ \xi_1,\xi_2,\xi_3 $ of a 3-Sasakian manifold $ M $ generate a 3-dimensional foliation $ \mathcal{F} $, whose space of leaves $ M/\mathcal{F} $ is a \emph{positive} (i.e.~its scalar curvature is positive) qK orbifold. \\
If $ M $ is a homogeneous 3-Sasakian manifold, then $ M \to M/\mathcal{F} $ is a locally trivial Riemannian fibration over a homogeneous positive qK manifold with fiber $ Sp(1) $ or $ SO(3) $.
\end{proposition}

\section{History of the Classification} \label{hist}

The earliest result concerning our topic was the classification for the related notion of (compact simply connected) homogeneous \emph{complex contact manifolds} (so-called \emph{$ \mathcal{C} $-spaces}) by \textsc{Boothby} in 1961 \cite{Boot}, who showed that these are in one-to-one correspondence to simple complex Lie algebras. The much more famous next step was the work of \textsc{Wolf} and \textsc{Alekseevskii} on qK manifolds in the 1960s. First, \textsc{Wolf} showed in 1961 that there is a one-to-one correspondence between $ \mathcal{C} $-spaces and compact simply connected \emph{symmetric} positive qK manifolds \cite[Theorem 6.1]{Wolf}. \\
\textsc{Boothby} and \textsc{Wolf} already emphasized the importance of the maximal root in the root system of a simple Lie algebra, which will also play a key role in our construction: \textsc{Wolf} demonstrated that the compact simply connected symmetric positive qK manifolds are precisely of the form $ G/N_G(K) $, where $ G $ is a compact simple Lie group and $ N_G(K) $ denotes the normalizer of the subgroup $ K $ corresponding to the compact real form of the subalgebra generated by the root spaces of a maximal root and its negative. As we will show in this article, the simply connected homogeneous 3-Sasakian manifolds are of the form $ G/(C_G(K))_0 $, where $ (C_G(K))_0 $ is the identity component of the centralizer $ C_G(K) $ of $ K $ in $ G $. \\
The manifolds of the type $ G /N_G(K) $ became known as \textsc{Wolf} spaces and \textsc{Alekseevskii} proved in 1968 that, in fact, all compact simply connected \emph{homogeneous} positive qK manifolds are of this form \cite[Theorem 1]{Alek}. We remark that \textsc{Alekseevskii} invoked the heavy \emph{Yano-Bochner theorem} several times throughout his proofs. \\
By 1994, \textsc{Boyer}, \textsc{Galicki} and \textsc{Mann} transferred these results to the 3-Sasakian realm \cite{BGM}. They combined the classification of homogeneous positive qK manifolds with \autoref{fibr} to obtain the following diffeomorphism type classification: \vspace{-.2cm}

\begin{theorem}[{\cite[Theorem C]{BGM}}] \label{bgm}
Every homogeneous 3-Sasakian manifold $ M = G/H $ (not necessarily simply connected) is precisely one of the following:
\begin{align*}
\frac{Sp(n+1)}{Sp(n)} \cong S^{4n+3}\, &, \quad \frac{Sp(n+1)}{Sp(n) \times \Z_2} \cong  \R P^{4n+3}\, , \quad \frac{SU(m)}{S(U(m-2) \times U(1))} \, , \\
\frac{SO(k)}{SO(k-4) \times Sp(1)}\, &, \quad \frac{G_2}{Sp(1)} \, , \quad \frac{F_4}{Sp(3)} \, , \quad \frac{E_6}{SU(6)} \,, \quad \frac{E_7}{Spin(12)} \, , \quad \frac{E_8}{E_7} \, .
\end{align*}
To avoid redundancy, we need to assume $ n \geq 0 $, $ m \geq 3 $ and $ k \geq 7 $.
\end{theorem} \vspace{-.2cm}

They also provided a more precise description of the 3-Sasakian structures in the four classical cases via \emph{3-Sasakian reduction} \cite{BGM}. \\
In 1996, \textsc{Bielawski} \cite{Biel} described the Riemannian structure on these spaces uniformly. Both for his result and for several later discussions, we need to recall the following construction: As was first described systematically by \textsc{Kobayashi} and \textsc{Nomizu} \cite{KN}, the study of $ G $-invariant geometric objects on a reductive homogeneous space \mbox{$ M = G/H = G/G_p $} can be greatly simplified by instead considering $ \Ad(H) $-invariant algebraic objects on a fixed reductive complement $ \mf{m} $ of $ \mf{h} $ in $ \mf{g} $. More precisely, the map $ \psi : \mf{m} \to T_pM, \, X \mapsto \overline{X}_p $ (where $ \overline{X}_p $ denotes the fundamental vector field of the left $ G $-action at $ p $) is an isomorphism that allows us to translate between $ \Ad(H) $-invariant tensors on $ \mf{m} $ and the restriction of $ G $-invariant tensor fields to $ T_pM $. \\
While actually working on a more algebro-geometric problem (singularities of nilpotent varieties) and employing very different methods (e.g.~Nahm's differential equation), \textsc{Bielawski} obtained the following

\begin{theorem}[{\cite[Theorem 4]{Biel}}]
For every homogeneous 3-Sasakian manifold $ M = G/H $ with reductive decomposition $ \mf{g} = \mf{m} \oplus \mf{h} $, there is a natural decomposition $ \mf{m} = \mf{sp}(1) \oplus \mf{m}' $ such that the metric on $ M $ corresponds to an inner product on $ \mf{m} $ of the form
\[ (X,Y) \mapsto - c \, B(X_{\mf{sp}(1)}, Y_{\mf{sp}(1)}) - \frac{c}{2} \, B(X_{\mf{m}'},Y_{\mf{m}'}) \, ,  \]
where $ B $ denotes the Killing form of $ \mf{g} $ and $ c > 0 $ is some constant.
\end{theorem}

In 2020, the work of \textsc{Draper}, \textsc{Ortega} and \textsc{Palomo} gave a new hands-on description of homogeneous 3-Sasakian manifolds \cite{DOP}. Their study was based on the following

\begin{definition}[{\cite[Definition 4.1]{DOP}}] \label{3Sd}
A \emph{3-Sasakian datum} is a pair $ (\mf{g},\mf{h}) $ of real Lie algebras such that
\begin{enumerate}
\item $ \mf{g} = \mf{g}_0 \oplus \mf{g}_1 $ is a $ \Z_2 $-graded compact simple Lie algebra whose even part is a sum of two commuting subalgebras,
\[ \mf{g}_0 = \mf{sp}(1) \oplus \mf{h} \, ; \]
\item there exists an $ \mf{h}^\C $-module $ W $ such that the complexified $ \mf{g}_0^\C $-module $\mathfrak{g}_1^\C $ is isomorphic to the tensor product of the natural $ \mf{sp}(1)^\C = \mathfrak{sl}(2,\mathbb{C})$-module $\mathbb{C}^2$ and $W$:
\[ \mf{g}_1^\C \cong \C^2 \otimes W \, . \]
\end{enumerate}
\end{definition}

Their main result is the following

\begin{theorem}[{\cite[Theorem 4.2]{DOP}}] \label{dop}
Let $ M = G/H $ be a homogeneous space such that $ H $ is connected and the Lie algebras $ (\mf{g},\mf{h}) $ constitute a 3-Sasakian datum. Consider the reductive complement $ \mf{m} := \mf{sp}(1) \oplus \mf{g}_1 $ and let $ X_1, X_2, X_3 \in \mf{m} $ denote the standard basis of $ \mf{sp}(1) $ and $ \xi_1, \xi_2, \xi_3 $ the corresponding $ G $-invariant vector fields on $ M $. If $ g $ and $ \varphi_i $ are the Riemannian metric and endomorphism fields described in \autoref{goal} and $ \eta_i = g(\xi_i,\cdot) $, then the tuple $ (M,g,\xi_i,\eta_i, \varphi_i)_{i=1,2,3} $ constitutes a homogeneous 3-Sasakian structure.
\end{theorem}

Furthermore, they conducted a case-by-case study to show that every compact simple Lie algebra admits a 3-Sasakian datum, thus providing a detailed analysis of one homogeneous 3-Sasakian structure (it is, at this point, not clear if there could be more than one such structure on a given space) on each of the diffeomorphism types discovered by \textsc{Boyer}, \textsc{Galicki} and \textsc{Mann}. \\

We finish this section by giving an overview of the structure of our proof of \autoref{goal}: In \Cref{const}, we first describe a way to construct a simply connected homogeneous 3-Sasakian manifold from a simple complex Lie algebra $ \mf{u} $ and a maximal root $ \alpha $ of $ \mf{u} $. More precisely, we first utilize the theory of root systems to generate a complexified version $ (\mf{u}, \mf{v}) $ of a 3-Sasakian datum. We then pass to the compact real forms $ (\mf{g},\mf{h}) $ to obtain a ``real'' 3-Sasakian datum in the sense of \autoref{3Sd} and apply \autoref{dop}. \\
For the converse argument in Sections \ref{univ} and \ref{G2}, we start with a simply connected homogeneous 3-Sasakian manifold $ M = G/H $, where $ G $ is a compact simply connected Lie group acting almost effectively and transitively on $ M $ via 3-Sasakian automorphisms. We prove that the Lie algebra $ \mf{g} $ and its complexification $ \mf{u} = \mf{g}^\C $ are simple and that the 3-Sasakian structure gives rise to a maximal root $ \alpha $ of $ \mf{u} $. We can therefore apply the previous construction and then show that this yields the same 3-Sasakian structure that we started with. \\
\Cref{AutM} completes the proof of \autoref{goal} by showing that no subgroup of $ \mathrm{Aut}_0(M) $ can act transitively. In particular, this proves that any two homogeneous $3$-Sasakian manifolds $M=G/H$, $M^\prime=G^\prime/H^\prime$ associated with two different simple complex Lie algebras $\mf{g}^\mathbb{C}\neq (\mf{g}^\prime)^\mathbb{C}$ are not isomorphic.

\section{Root System Preliminaries} \label{roots}

Later on, we will need certain basic facts about root systems, which we decided to collect in this section: Let $ \mf{u} $ be a (finite-dimensional) semisimple complex Lie algebra. Its Killing form is non-degenerate and thus gives rise to an isomorphism $ \mf{u} \to \mf{u}^* $ and a non-degenerate, symmetric bilinear form $\langle \cdot , \cdot \rangle $ on $ \mf{u}^* $. We fix a Cartan subalgebra $ \mf{c} \subset \mf{u} $ and denote the corresponding root system and root spaces by $ \Phi \subset \mf{c}^* $ and $ \mf{u}_\alpha \subset \mf{u} $ for $ \alpha \in \mf{c}^* $, respectively. \\
Each root $ \alpha \in \Phi $ has an associated coroot $ H_\alpha \in \mf{c} $ defined as the unique element of $ [\mf{u}_\alpha, \mf{u}_{-\alpha}] $ satisfying $ \alpha(H_\alpha) = 2 $. Furthermore, $ \mf{s}_\alpha := \mf{u}_\alpha \oplus \mf{u}_{-\alpha} \oplus [\mf{u}_\alpha, \mf{u}_{-\alpha}] $ is a subalgebra of $ \mf{u} $ which is isomorphic to $ \mf{sl}(2,\C) $. This isomorphism can be made explicit by choosing an \emph{$ \mf{sl}_2 $-triple}, i.e.~vectors $ X_\alpha \in \mf{u}_\alpha, Y_\alpha \in \mf{u}_{-\alpha} $ satisfying the commutation relations
\[ [H_\alpha,X_\alpha] = 2X_\alpha \, , \quad [H_\alpha,Y_\alpha] = -2Y_\alpha \, , \quad [X_\alpha,Y_\alpha] = H_\alpha \, . \]
Moreover, it can be shown that for any root $ \alpha \in \Phi $ and any linear form $ \beta \in \mf{c}^* $:
\[ c_{\alpha\beta} := \beta(H_\alpha) = \frac{2\langle \beta, \alpha\rangle}{\langle \alpha,\alpha \rangle} \, . \]
In particular, $ c_{\alpha\beta} = 0 $ if and only if $ \alpha $ and $ \beta $ are perpendicular to each other (with respect to $ \langle \cdot, \cdot \rangle $). In case $ \beta $ is also a root, $ c_{\alpha\beta} $ is an integer, which we will call the \emph{Cartan number of $ \beta $ with respect to $ \alpha $}. Fixing a root $ \alpha \in \Phi $, we can therefore decompose
\[ \mf{u} = \bigoplus_{k \in \Z} \mf{u}^{(k)} \, , \quad \textnormal{where} \quad \mf{u}^{(k)} := \bigoplus_{\substack{\beta \in \mf{c}^* \, , \\ c_{\alpha\beta} = k}} \mf{u}_\beta \, . \]
Since $ c_{\alpha\beta} $ is linear in $ \beta $, this decomposition is in fact a $ \Z $-grading, i.e.~$ [\mf{u}^{(k)}, \mf{u}^{(\ell)}] \subset \mf{u}^{(k+\ell)} $. We also note that $ \mf{u}^{(k)} $ is precisely the $ k $-eigenspace of $ \ad(H_\alpha) $. One can visualize this gra- ding using parallel copies of hyperplanes perpendicular to $ \alpha $, \mbox{e.g.~for the root system $ A_2 $:}

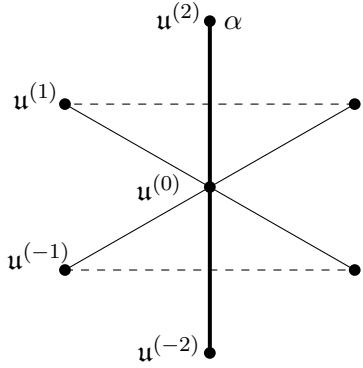
\begin{figure}[H]
\begin{center}
\begin{tikzpicture}[scale=2.2]
\draw[line width=1.5pt] (90:1cm)--(0,0)-- (270:1cm);
\draw (30:1cm)--(0,0)--(210:1cm);
\draw (150:1cm)--(0,0)--(330:1cm);

\fill  (0,0) circle (1pt);
\fill  (30:1cm) circle (1pt);
\fill  (90:1cm) circle (1pt);
\fill  (150:1cm) circle (1pt);
\fill  (210:1cm) circle (1pt);
\fill  (270:1cm) circle (1pt);
\fill  (330:1cm) circle (1pt);

\draw[dashed] (30:1cm)--(150:1cm);
\node at (82:1cm) {$\alpha$};
\node at (100:1.05cm) {$\mf{u}^{(2)}$};
\node at (152:1.175cm) {$\mf{u}^{(1)}$};
\node at (180:0.3cm) {$\mf{u}^{(0)}$};
\node at (202:1.1cm) {$\mf{u}^{(-1)}$};
\draw[dashed] (210:1cm)--(330:1cm);
\node at (256:1cm) {$\mf{u}^{(-2)}$};
\end{tikzpicture}
\end{center}
\vfill
\end{figure}

The structure of this grading is related to the notion of \emph{maximality} of the root $ \alpha $: \mbox{Assuming} we have chosen a set $ \Delta \subset \Phi $ of simple roots, we may introduce a partial order $ \leq $ on $ \Phi $ by stipulating that $ \alpha \leq \beta $ if and only if $ \beta - \alpha $ is a linear combination of roots in $ \Delta $ with non-negative coefficients. A root $ \alpha \in \Phi $ is called \emph{maximal} if there is a choice of simple roots such that there is no strictly larger root than $ \alpha $ with respect to the induced partial order. The following lemma was adapted from \cite[Theorem 4.2]{Wolf}:

\begin{lemma} \label{maxroot}
For any root $ \alpha \in \Phi $, the following statements are equivalent:
\begin{enumerate}
\item[i)] $ \alpha $ is maximal.
\item[ii)] $ |c_{\alpha\beta}| \leq 2 $ for all roots $ \beta \in \Phi $ and $ c_{\alpha\beta} = \pm 2 $ if and only if $ \beta = \pm \alpha $.
\end{enumerate}
\end{lemma}
\begin{proof} $i) \implies ii)$: It is well-known that for $ \beta \in \Phi \setminus \{\pm\alpha\} $, the Cartan number is given by $ c_{\alpha\beta} = p-q $, where $ p,q \in \N_0 $ are the greatest non-negative integers such that $ \beta + r \alpha \in \Phi $ for every $ r \in \{-p, \ldots, q\} $ \cite[Proposition 2.29]{Knap}. Suppose there was some $ \beta \in \Phi \setminus \{\alpha\} $ such that $ c_{\alpha\beta} \geq 2 $. Then $ p \geq 2 $, so that \mbox{$ \beta - \alpha, \beta -2\alpha \in \Phi $} and their negatives $ \alpha - \beta, 2\alpha -\beta \in \Phi $ are roots. In fact, $ \alpha - \beta $ has to be a \emph{non-negative} linear combination of simple roots (for some choice of simple roots with respect to which $ \alpha $ is maximal), since otherwise $ \beta > \alpha $. But then $ 2\alpha -\beta \geq \alpha $ and maximality of $ \alpha $ would imply $ 2 \alpha - \beta = \alpha $, i.e.~$ \beta = \alpha $. For $ \beta \in \Phi \setminus \{ - \alpha\} $ such that $ c_{\alpha\beta} \leq -2 $, we apply this argument to $ -\beta $. \\
$ ii) \implies i)$: We may choose a set of simple roots $\Delta $ in such a way that $ c_{\alpha\beta} \geq 0 $ for all $ \beta \in \Delta $. This can be achieved by first choosing positive roots using a slight perturbation of the hyperplane perpendicular to $ \alpha $. Let $ \beta \in \Phi $ such that $ \beta \geq \alpha $, i.e.~$ \beta-\alpha = \sum_{i=1}^n \lambda_i \alpha_i $, where $ \lambda_i \geq 0 $ and $ \alpha_i \in \Delta $. Then,
\[ c_{\alpha\beta} = c_{\alpha\alpha} + c_{\alpha(\beta-\alpha)} = 2 + \sum_{i=1}^n \lambda_i \underbrace{c_{\alpha\alpha_i}}_{\geq 0} \geq 2 \, . \]
Hypothesis $ ii) $ then implies $ \beta = \alpha $, so that $ \alpha $ is maximal.
\end{proof}

Finally, we remark that in an \emph{irreducible} root system $ \Phi $, the maximal root is unique up to the action of the Weyl group: This follows because in an irreducible root system, the maximal root is uniquely determined after choosing simple roots, and any two choices of simple roots can be mapped to each other by the Weyl group.

\section{Constructing Homogeneous 3-Sasakian Manifolds from Simple Lie Algebras} \label{const}

Our goal in this section is the following construction:

\begin{theorem} \label{constthm}
Let $ \mf{u} $ be a simple complex Lie algebra, $ \alpha $ a maximal root in its root system, $ \mf{g} $ the compact real form of $ \mf{u} $ and $ \mf{k} \cong \mf{sp}(1) $ the compact real form of the subalgebra $ \mf{s}_\alpha = \mf{u}_\alpha \oplus \mf{u}_{-\alpha} \oplus [\mf{u}_\alpha,\mf{u}_{-\alpha}] \cong \mf{sl}(2,\C) $. Let $ G $ denote the simply connected Lie group with Lie algebra $ \mf{g} $, $ K $ the connected subgroup with Lie algebra $ \mf{k} $ and $ H = (C_G(K))_0 $ the identity component of the centralizer $ C_G(K) $ of $ K $ in $ G $. Then, the simply connected homogeneous space $ M = G/H $ admits a homogeneous 3-Sasakian structure whose tensors are given by \autoref{goal}. All possible choices of a maximal root lead to isomorphic 3-Sasakian manifolds.
\end{theorem}

\begin{definition}\label{c3Sd}
A \emph{complex 3-Sasakian datum} is a pair $(\mf{u},\mf{v})$ of complex Lie algebras such that
\begin{enumerate}
\item $ \mf{u} = \mf{u}_0 \oplus \mf{u}_1 $ is a $ \Z_2 $-graded simple Lie algebra whose even part is a sum of two commuting subalgebras,
\[ \mf{u}_0=\mf{v}\oplus\mathfrak{sl}(2,\mathbb{C}) \, ; \]
\item there exists a $ \mf{v} $-module $ W $ such that $\mf{u}_1\cong\mathbb{C}^2\otimes W$ as $ \mf{u}_0 $-modules.
\end{enumerate}
\end{definition}

\begin{remark} \label{constrem}
We formulated the above definition in the given way because it allows us to branch off into two cases: Our primary interest in this article will be to consider the \emph{compact} real forms $ (\mf{g},\mf{h}) $ of $ (\mf{u}, \mf{v}) $ which then form a 3-Sasakian datum in the sense of \autoref{3Sd}. On the other hand, one may also look at the real form $(\mathfrak{g}^*,\mathfrak{h})$ of $ (\mf{u}, \mf{v}) $ given by $\mathfrak{g}^*=\mathfrak{h}\oplus\mathfrak{s}_\alpha\oplus i\mathfrak{g}_1$, to obtain a \emph{generalized 3-Sasakian datum} in the sense of \cite{ADS}. These give rise to homogeneous negative 3-$ (\alpha,\delta) $-Sasakian manifolds by a construction similar to \autoref{dop}, compare \cite[Theorem 3.1.1]{ADS}.
\end{remark}

\begin{proposition}\label{rootconst}
Let $ \mf{u} $ be a simple complex Lie algebra and $ \alpha \in \Phi $ a maximal root in its root system. Set $ \Phi_0 := \{\beta \in \Phi \, | \, c_{\alpha\beta} = 0\} $ as well as
\[ \mf{v} := \ker \alpha \oplus \bigoplus_{\beta \in \Phi_0} \mf{u}_\beta \, . \]
Then, $(\mf{u},\mf{v})$ is a complex 3-Sasakian datum.
\end{proposition}
\begin{proof}  \vspace{-.5cm} Using the $ \Z $-grading from \Cref{roots}, we let
\[ \mf{u}_0 := \mf{u}^{(-2)} \oplus \mf{u}^{(0)} \oplus \mf{u}^{(2)} \, , \quad \mf{u}_1 := \mf{u}^{(-1)} \oplus \mf{u}^{(1)} \, . \]
Since $ |c_{\alpha\beta}| \leq 2 $ for all $ \beta \in \Phi $ by \autoref{maxroot}, we have $ \mf{u} = \mf{u}_0 \oplus \mf{u}_1 $. Because $ \mf{u}_0 $ and $ \mf{u}_1 $ are comprised of the $ \mf{u}^{(k)} $ with even and odd $ k $ respectively, this decomposition is in fact a $ \Z_2 $-grading. We claim that
\[ \mf{u}_0 = \mf{s}_\alpha \oplus \mf{v} \]
as a direct sum of Lie algebras, where $ \mf{s}_\alpha = \mf{u}_\alpha \oplus \mf{u}_{-\alpha} \oplus [\mf{u}_\alpha, \mf{u}_{-\alpha}] $. Since $ c_{\alpha\beta} = \pm 2 $ if and only if $ \beta = \pm \alpha $, we have the following vector space decompositions:
\[ \mf{u}_0 = \mf{u}_\alpha \oplus \mf{u}_{-\alpha} \oplus \mf{c} \oplus \bigoplus_{\beta \in \Phi_0} \mf{u}_\beta \, , \quad \mf{c} = [\mf{u}_\alpha,\mf{u}_{-\alpha}] \oplus \ker \alpha \, . \]
In order to show that $ \mf{v} $ is indeed a sub\emph{algebra} of $ \mf{u} $, note that $ [\mf{u}_\beta, \mf{u}_\gamma] \subset \mf{u}_{\beta+\gamma} $ for any $ \beta, \gamma \in \Phi_0 $. Now if $ \beta + \gamma $ is a root, then $ \beta + \gamma \in \Phi_0 $, so $ \mf{u}_{\beta + \gamma} \subset \mf{v} $. If $ \beta + \gamma $ is not a root and not zero, then $ \mf{u}_{\beta+\gamma} = 0 \subset \mf{v} $. If $ \beta + \gamma = 0 $, then $ [\mf{u}_\beta, \mf{u}_{-\beta}] = \langle H_\beta \rangle \subset \ker \alpha $ because $ \beta \in \Phi_0 $. To check that $ \mf{s}_\alpha $ and $ \mf{v} $ commute, we recall that $ \mf{v} $ is a subset of $ \mf{u}^{(0)} = \ker \ad_{H_\alpha} $. For $ \beta \in \Phi_0 $, we have $ [\mf{u}_{\pm \alpha}, \mf{u}_\beta] \subset \mf{u}_{\pm \alpha+\beta} \subset \mf{u}^{(\pm 2)} = \mf{u}_{\pm \alpha} $, so $ \mf{u}_{\pm \alpha} $ and $ \mf{u}_\beta $ commute. \\
We now verify the second condition from \autoref{c3Sd} for the $ \mf{v} $-module $ W := \mf{u}^{(1)} $: We choose an $ \mf{sl}_2 $-triple $ (X_\alpha, Y_\alpha, H_\alpha) $ and consider the following linear map:
\begin{align*} \setlength{\tabcolsep}{2pt} \begin{tabular}{c c c c c c c} $ \Psi: $ & $ \mf{u}_1 $ & $=$ & $ \mf{u}^{(-1)} \oplus \mf{u}^{(1)} $ & $\to $ & $ \C^2 \otimes W $ & , \\
& $ X $ & $=$ & $ X^{(-1)} + X^{(1)} $ & $ \mapsto $ & $ (1,0) \otimes X^{(1)} + (0,1) \otimes [X_\alpha,X^{(-1)}] $ & .
\end{tabular} \end{align*}
If $ \beta \in \Phi $ such that $ c_{\alpha\beta} = -1 $, then $ \beta + \alpha $ must be a root and $ [\mf{u}_\alpha, \mf{u}_\beta] = \mf{u}_{\alpha+\beta} $. This shows that $ \ad_{X_\alpha}: \mf{u}_\beta \to \mf{u}_{\beta+\alpha} $ and, by extension, $ \Psi $ are linear isomorphisms. It remains to be shown that $ \Psi $ preserves the $ \mf{u}_0 $-module structure: \\
If $ Z \in \mf{v} \subset \mf{u}^{(0)} $, then $ \ad_Z $ preserves the decomposition \mbox{$ \mf{u}_1 = \mf{u}^{(-1)} \oplus \mf{u}^{(1)} $.} Since $ \mf{v} $ and $ \mf{s}_\alpha $ are commuting subalgebras of $ \mf{u} $, so are their respective adjoint subrepresentations,
\begin{align*}
\Psi ([Z,X]) &= (1,0) \otimes [Z,X^{(1)}] + (0,1) \otimes [X_\alpha, [Z,X^{(-1)}]] \\
&= (1,0) \otimes [Z,X^{(1)}] + (0,1) \otimes [Z, [X_\alpha,X^{(-1)}]] = Z \cdot \Psi(X) \, . 
\end{align*}
Here, $ \cdot $ denotes the adjoint representation of $ \mf{v} $ on $ W $, while in the following equations it will signify the standard representation of $ \mf{sl}(2,\C) $ on $ \C^2 $. Finally, we check the representation of the basis $ (X_\alpha, Y_\alpha, H_\alpha) $ of $ \mf{s}_\alpha $:
\[ \Psi([X_\alpha,X]) = \Psi([X_\alpha, X^{(-1)}]) = (1,0) \otimes [X_\alpha,X^{(-1)}] = X_\alpha \cdot \Psi(X) \, . \]
By the Jacobi identity,
\begin{align*}
\Psi([Y_\alpha,X]) &= \Psi([Y_\alpha, X^{(1)}]) = (0,1) \otimes [X_\alpha,[Y_\alpha,X^{(1)}]] \\
&= (0,1) \otimes \big([\underbrace{[X^{(1)},X_\alpha]}_{=0}, Y_\alpha] + [\underbrace{[X_\alpha, Y_\alpha]}_{=H_\alpha}, X^{(1)}]\big) \\
&= (0,1) \otimes X^{(1)} = Y_\alpha \cdot \Psi(X) \, .
\end{align*}
Ultimately,
\[ \Psi([H_\alpha, X]) = \Psi(X^{(1)} - X^{(-1)}) = (1,0) \otimes X^{(1)} + (0,-1) \otimes [X_\alpha,X^{(-1)}] = H_\alpha \cdot \Psi(X)\, . \qedhere \]
\end{proof}

\begin{proof}[Proof of \autoref{constthm}] Starting from a simple complex Lie algebra $ \mf{u} $ and a maximal root $ \alpha $, \autoref{rootconst} yields a complex 3-Sasakian datum $ (\mf{u},\mf{v}) $. As mentioned in \autoref{constrem}, the compact real forms $ (\mf{g}, \mf{h}) $ constitute a ``real'' 3-Sasakian datum in the sense of \autoref{3Sd} and \autoref{dop} endows $ M = G/H $ with a homogeneous 3-Sasakian structure. Since $ \mf{v} = C_\mf{u}(\mf{s}_\alpha) $ and thus $ \mf{h} = C_\mf{g}(\mf{k}) $, it follows that $ H = (C_G(K))_0 $.
\end{proof}

\begin{example}
Let us illustrate the construction using the special case $ \rk \mf{u} = 2 $: Here, the only simple Lie algebras are $ \mf{sl}(3,\C) $, $ \mf{sp}(4,\C) $ and $ \mf{g}_2 $, corresponding (in order) to the root systems $ A_2, C_2 $ and $ G_2 $. The following diagrams depict the subalgebras $ \mf{v} $ and $ \mf{s}_\alpha $ from the proposition in these three cases:

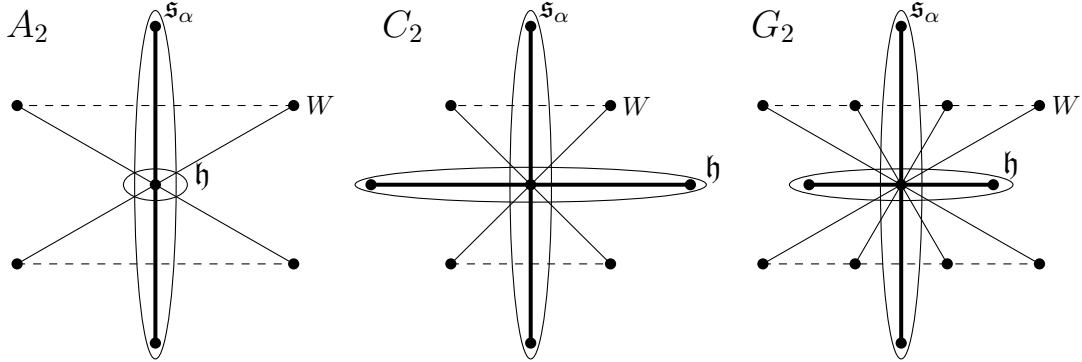
\begin{figure}[H]
\begin{center}
\begin{tikzpicture}[scale=2.1]
\draw[line width=1.5pt] (90:1cm)--(0,0)-- (270:1cm);
\draw (30:1cm)--(0,0)--(210:1cm);
\draw (150:1cm)--(0,0)--(330:1cm);

\fill  (0,0) circle (1pt);
\fill  (30:1cm) circle (1pt);
\fill  (90:1cm) circle (1pt);
\fill  (150:1cm) circle (1pt);
\fill  (210:1cm) circle (1pt);
\fill  (270:1cm) circle (1pt);
\fill  (330:1cm) circle (1pt);

\draw[dashed] (30:1cm)--(150:1cm);
\node at (10:0.3cm) {\large $\mathfrak{h}$};
\node at (82:1.1cm) {\large $\mathfrak{s}_\alpha$};
\node at (26:1.15cm) {\small $W$};
\draw[dashed] (210:1cm)--(330:1cm);
\draw (0:0) circle (0.2cm and 0.1cm);
\draw (0,0) circle (0.13cm and 1.1cm);
\node at (-0.8cm,1cm) {\Large $A_2$};
\end{tikzpicture}
\begin{tikzpicture}[scale=2.1]
\draw[line width=1.5pt] (0:1cm)--(0,0)--(180:1cm);
\draw[line width=1.5pt] (90:1cm)--(0,0)-- (270:1cm);
\draw (45:0.707cm)--(0,0)--(225:0.707cm);
\draw (135:0.707cm)--(0,0)--(315:0.707cm);

\fill  (0,0) circle (1pt);
\fill  (0:1cm) circle (1pt);
\fill  (45:0.707cm) circle (1pt);
\fill  (90:1cm) circle (1pt);
\fill  (135:0.707cm) circle (1pt);
\fill  (180:1cm) circle (1pt);
\fill  (225:0.707cm) circle (1pt);
\fill  (270:1cm) circle (1pt);
\fill  (315:0.707cm) circle (1pt);

\draw[dashed] (45:0.707cm)--(135:0.707cm);
\node at (5:1.15cm) {\large $\mathfrak{h}$};
\node at (82:1.1cm) {\large $\mathfrak{s}_\alpha$};
\node at (37:0.83cm) {\small $W$};
\draw[dashed] (225:0.707cm)--(315:0.707cm);
\draw (0:0) circle (1.1cm and 0.11cm);
\draw (0,0) circle (0.13cm and 1.1cm);
\node at (-0.8cm,1cm) {\Large $C_2$};
\end{tikzpicture}
\begin{tikzpicture}[scale=2.1]
\draw[line width=1.5pt] (0:0.577cm)--(0,0)--(180:0.577cm);
\draw[line width=1.5pt] (90:1cm)--(0,0)-- (270:1cm);
\draw (30:1cm)--(0,0)--(210:1cm);
\draw (60:0.577cm)--(0,0)--(240:0.577cm);
\draw (120:0.577cm)--(0,0)--(300:0.577cm);
\draw (150:1cm)--(0,0)--(330:1cm);

\fill  (0,0) circle (1pt);
\fill  (0:0.577cm) circle (1pt);
\fill  (60:0.577cm) circle (1pt);
\fill  (120:0.577cm) circle (1pt);
\fill  (180:0.577cm) circle (1pt);
\fill  (240:0.577cm) circle (1pt);
\fill  (300:0.577cm) circle (1pt);
\fill  (30:1cm) circle (1pt);
\fill  (90:1cm) circle (1pt);
\fill  (150:1cm) circle (1pt);
\fill  (210:1cm) circle (1pt);
\fill  (270:1cm) circle (1pt);
\fill  (330:1cm) circle (1pt);

\draw[dashed] (30:1cm)--(150:1cm);
\node at (10:0.7cm) {\large $\mathfrak{h}$};
\node at (82:1.1cm) {\large $\mathfrak{s}_\alpha$};
\node at (26:1.15cm) {\small $W$};
\draw[dashed] (210:1cm)--(330:1cm);
\draw (0:0) circle (0.7cm and 0.1cm);
\draw (0,0) circle (0.13cm and 1.1cm);
\node at (-0.8cm,1cm) {\Large $G_2$};
\end{tikzpicture}
\end{center}
\vfill
\end{figure}

The corresponding homogeneous 3-Sasakian manifolds are (in order) the Aloff-Wallach space $W^{1,1}=SU(3)/S^1$, the 7-sphere $S^7=Sp(2)/Sp(1)$ and the exceptional space $G_2/Sp(1)$.
\end{example}

We finish this section by showing that the maximal root is in fact an auxiliary choice:

\begin{lemma}
All possible choices of a maximal root in \cref{rootconst} lead to isomorphic 3-Sasakian manifolds.
\end{lemma}
\begin{proof}
Let $ \mf{u} $ be a simple complex Lie algebra, $ \mf{g} $ its compact real form, $ G $ the corresponding simply connected Lie group and $ T \subset G $ a maximal torus. Let $ \alpha, \wt{\alpha} $ denote two maximal roots in the root system $ \Phi $ of $ \mf{u} $ with respect to the Cartan subalgebra given by the complexification of the Lie algebra of $ T $. As mentioned at the end of \Cref{roots}, the maximal root of $ \Phi $ is unique up to the action of the Weyl group $ W(G) = N_G(T)/T $, so there is a representative $ w \in N_G(T) $ such that $ \Ad_w^\C(H_\alpha) = H_{\wt{\alpha}} $. \\
Because the Weyl group acts orthogonally on the root system, $ \Ad_w^\C:\mf{u} \to \mf{u} $ maps the $ \Z $-grading $ \mf{u}^{(k)} $ with respect to $ \alpha $ to the grading $ \wt{\mf{u}}^{(k)} $ with respect to $ \wt{\alpha} $. This implies that $ \Ad_w \, \mf{h} = \wt{\mf{h}} $, where $ \mf{h}, \wt{\mf{h}} \subset \mf{g} $ are the compact real forms of the subalgebras \mbox{$ \mf{v}, \wt{\mf{v}} \subset \mf{u} $} considered in \autoref{rootconst}. Consequently, $ wHw^{-1} =\wt{H} $ for the corresponding connected subgroups $ H, \wt{H} \subset G $ and we have a well-defined diffeomorphism $ G/H \to G/\wt{H} $, $ gH \mapsto wgw^{-1}\wt{H} $. One easily checks from the definitions in \autoref{goal} that this map is a 3-Sasakian isomorphism.
\end{proof}

\section{Deconstructing Homogeneous 3-Sasakian Manifolds} \label{univ}

This section is the centerpiece of the article, where we explain a crucial step in the proof of \autoref{goal}, namely:

\begin{theorem} \label{conv}
Every simply connected homogeneous 3-Sasakian manifold arises from the construction described in \Cref{const}.
\end{theorem} \vspace{-.125cm}

From now on, let $ (M^{4n+3},g,\xi_i,\eta_i,\varphi_i)_{i=1,2,3} $ denote a simply connected homogeneous 3-Sasakian manifold and let $ G $ be a compact simply connected Lie group acting almost effectively (i.e.~the kernel of the action is discrete and hence finite) and transitively on $ M $ by 3-Sasakian automorphisms. We will show that the Lie algebra $ \mf{g} $ of $ G $ and its complexification $ \mf{u} = \mf{g}^\C $ are simple and describe how the 3-Sasakian structure gives rise to a maximal root $ \alpha $ of $ \mf{u} $ with respect to a suitably chosen Cartan subalgebra. We can then apply the construction from \Cref{const} and prove that this yields the same 3-Sasakian structure that we started with. \\
The prototypical example to have in mind is where $ G $ is the universal cover of $ \Aut_0(M) $, the identity component of the 3-Sasakian automorphism group of $ M $. By \autoref{Einstein}, $ M $ is compact, so by the Myers-Steenrod theorem, the isometry group $ \Isom(M) $ of $ M $ is a compact Lie group. The subgroup $ \Aut(M) \subset \Isom(M) $ of 3-Sasakian automorphisms of $ M $ is clearly closed and thus also a compact Lie group. Since $ M $ is connected, the identity component $ \Aut_0(M) $ still acts transitively. The universal cover of $ \Aut_0(M) $ acts almost effectively, transitively and by 3-Sasakian automorphisms. It will follow from the results that we are about to prove that the universal cover of $ \Aut_0(M) $ is also compact. \\
Later on, we will show that, in fact, the effectively acting quotient of any group $ G $ satisfying the above assumptions is automatically the full identity component $ \Aut_0(M) $ of the automorphism group. \\
Since $ G $ is compact, its Lie algebra $ \mf{g} $ is reductive, i.e.~decomposes as a direct sum of a semisimple subalgebra and its center $ Z(\mf{g}) $. We first show that $ \mf{g} $ itself is semisimple.

\begin{lemma} \label{lieder}
For $ X, Y \in \mf{g} $, the fundamental vector fields satisfy the equation
\[ d\eta_i(\overline{X},\overline{Y}) = \eta_i ([\overline{X}, \overline{Y}]) \, . \]
Notably, evaluating the left-hand side at a point $ p \in M $ depends on $ \overline{X}, \overline{Y}$ only through their values at $ p $, while the right-hand side a priori depends on the values in a neighborhood of $ p $.
\end{lemma}
\begin{proof}
The standard formula for the exterior derivative reads
\[ d\eta_i(\overline{X},\overline{Y}) = \overline{X}\big(\eta_i(\overline{Y})\big) - \overline{Y}\big(\eta_i(\overline{X})\big) - \eta_i ([\overline{X},\overline{Y}]) \, . \]
The Leibniz rule for the Lie derivative implies
\[ \overline{X}\big(\eta_i(\overline{Y})\big) = \mathcal{L}_{\overline{X}} \big(\eta_i(\overline{Y})\big) = \big(\mathcal{L}_{\overline{X}}\eta_i \big) (\overline{Y}) + \eta_i\big(\mathcal{L}_{\overline{X}}\overline{Y}\big) = \eta_i ([\overline{X},\overline{Y}])  \, , \]
where $ \mathcal{L}_{\overline{X}}\eta_i = 0 $ because $ G $ acts by 3-Sasakian automorphisms. Applying the same reasoning to the second term yields $ \overline{Y}\big(\eta_i(\overline{X})\big) = - \eta_i([\overline{X},\overline{Y}]) $. 
\end{proof}

\begin{proposition}
The Lie algebra $ \mf{g} $ has trivial center and is therefore semisimple.
\end{proposition}
\begin{proof}
Let $ X \in \mf{g} $ such that $ X \neq 0 $. Since $ G $ acts almost effectively, there is a point $ p \in M $ such that $ \overline{X}_p \neq 0 $ and thus an index $ i \in \{1,2,3\} $ such that $ \overline{X}_p $ is not proportional to $ (\xi_i)_p $. We show that there exists some $ Y \in \mf{g} $ satisfying $ \eta_i (\overline{[X,Y]}_p) \neq 0 $, which implies $ [X,Y] \neq 0 $: Because $ G $ acts transitively, we may choose some $ Y \in \mf{g} $ such that $ \overline{Y}_p = \varphi_i \overline{X}_p $. From the previous lemma, we have
\[ \eta_i (\overline{[X,Y]}_p ) = -d\eta_i(\overline{X}_p, \overline{Y}_p) = - d\eta_i(\overline{X}_p, \varphi_i \overline{X}_p) \, . \]
One of the Sasaki equations in \autoref{Sdef} reads $ \varphi_i^2 \overline{X}_p = - \overline{X}_p + P_i \overline{X}_p $, where $ P_i $ denotes the orthogonal projection to the line through $ (\xi_i)_p $. Hence,
\[\ \eta_i(\overline{[X,Y]}_p) = 2g_p\big(\overline{X}_p,\overline{X}_p - P_i\overline{X}_p\big) = 2 \big\|\overline{X}_p - P_i \overline{X}_p \big\|^2 \neq 0 \, .\qedhere \]
\end{proof}

\begin{remark}
The compactness assumption fails for homogeneous negative $3$-$(\alpha,\delta)$-Sasakian manifolds. Thus, unlike with the construction in the previous section, a classification cannot be achieved by the method described here. Indeed, in \cite{ADS} homogenoeus negative $3$-$(\alpha,\delta)$-Sasakian manifolds with a transitive action by a non-semisimple Lie group are constructed.
\end{remark}

Since $ \mf{g} $ is now both semisimple and the Lie algebra of a compact Lie group, its Killing form $ B $ is negative definite. We fix a point $ p \in M $ and let $ H := G_p $ denote its isotropy group. We write $ \theta: G \to M, \, g \mapsto g \cdot p $ for the orbit map, which has surjective differential $ d\theta_e: T_eG \cong \mf{g} \to T_pM, \, X \mapsto \overline{X}_p $. Let $ \alpha_i := \theta^*\eta_i $ denote the pullback of the contact form along the orbit map, which we may view - depending on the context - as either a linear form on $ \mf{g} $ or as a left-invariant differential one-form on $ G $. In their seminal 1958 article \cite{BW}, \textsc{Boothby} and \textsc{Wang} exhibited the following results:

\begin{lemma}[{\cite[Lemmata 3,4,5]{BW}}]
The one-form $ \alpha_i $ is $ \Ad(H) $-invariant, satisfies $ \alpha_i(\mf{h}) = 0 $ and $ d\alpha_i $ has rank $ 4n+2 $. Furthermore, the Lie algebra of the subgroup $ \{g \in G \; | \; \Ad_g^*\alpha_i = \alpha_i \} $ is given by $ \ker d\alpha_i $, contains $ \mf{h} $ and has dimension $ \dim \mf{h} + 1 $.
\end{lemma}

We now let $ \widetilde{X_i} \in \mf{g} $ denote the Killing dual of $ \alpha_i $, i.e.~$ B(\widetilde{X_i}, \cdot) = \alpha_i $ and consider \mbox{$ X_i := \widetilde{X_i}/B(\widetilde{X_i}, \widetilde{X_i}) $.} $ \Ad $-invariance of $ B $ implies that $ \{g \in G \; | \; \Ad_g X_i = X_i \} $ and $ \{g \in G \; | \; \Ad_g^*\alpha_i = \alpha_i\} $ coincide, so
\[ C_\mf{g}(X_i) = \ker d\alpha_i = \mf{h} \oplus \langle X_i \rangle \, . \]

\begin{proposition} \label{reeb} \vspace{-.125cm}
The fundamental vector fields $ \overline{X_i} $ coincide with the Reeb vector fields $ \xi_i $ at the point $ p $ and obey the same commutator relations $ [X_i,X_j] = 2\varepsilon_{ijk}X_k $, where $(i,j,k)$ is a permutation of $(1,2,3)$.
\end{proposition}
\begin{proof} \vspace{-.25cm}
Clearly, $ X_i \in C_\mathfrak{g}(X_i) = \ker d\alpha_i $, so that $ (\overline{X_i})_p \in \ker (d\eta_i)_p $. Furthermore, we have $ 1 = \alpha_i( X_i) = (\eta_i)_p(\overline{X_i})_p $. Thus, $ (\overline{X_i})_p $ statisfies the uniquely defining equations of the Reeb vector $ (\xi_i)_p $. Phrased differently, $ X_i $ (viewed as a left-invariant vector field on $ G $) and $ \xi_i $ are $ \theta $-related. Consequently, the Lie brackets $ [X_i,X_j] $ and $ [\xi_i, \xi_j] = 2\varepsilon_{ijk} \xi_k $ are also $ \theta $-related and in particular, $ \overline{[X_i, X_j]}_p = 2\varepsilon_{ijk} (\xi_k)_p = 2\varepsilon_{ijk} (\overline{X_k})_p $. Hence, $ [X_i,X_j] $ and $ 2\varepsilon_{ijk} X_k $ could only differ by an element of $ \mf{h} $. But $ B(X_k, \mf{h}) = \alpha_k(\mf{h}) = 0 $ and $ B([X_i,X_j], \mf{h}) = B(X_i,[X_j,\mf{h}]) = 0 $, so that also $ B([X_i,X_j]-2\varepsilon_{ijk}X_k, \mf{h}) = 0 $.
\end{proof}

Let $ \mf{s} $ be a maximal Abelian subalgebra of $ \mf{h} $. Since $ C_\mf{g}(X_1) = \mf{h} \oplus \langle X_1\rangle $, it follows that $ \mf{t} := \mf{s} \oplus \langle X_1 \rangle $ is a maximal Abelian subalgebra of $ \mf{g} $. In particular, $ \rk G = \rk H +1 $. The Riemannian metric $ g $ corresponds to an $ \Ad(H) $-invariant and thus also $ \ad(\mf{h}) $-invariant inner product on a reductive complement of our choice. The following lemma states that this inner product is even $ \ad(\mf{t}) $-invariant:

\begin{lemma} \label{adtinv}
For all $ Y,Z \in \mf{g} $, we have
\[ g_p(\overline{[X_i,Y]}_p,\overline{Z}_p) + g_p(\overline{Y}_p, \overline{[X_i,Z]}_p) = 0 \, . \]
\end{lemma}
\begin{proof}
Since $ \overline{X_i} $ is a Killing vector field ($ G $ acts isometrically) that coincides with $ \xi_i $ at $ p $, we obtain
\begin{align*}
g_p(\overline{[X_i,Y]}_p,\overline{Z}_p) + g_p(\overline{Y}_p, \overline{[X_i,Z]}_p) &= - g_p([\overline{X_i},\overline{Y}]_p,\overline{Z}_p) - g_p(\overline{Y}_p, [\overline{X_i},\overline{Z}]_p) \\
&= - (\overline{X_i})_p\big(g(\overline{Y},\overline{Z})\big) =  - (\xi_i)_p\big(g(\overline{Y},\overline{Z})\big) \, .
\end{align*}
Because the Levi-Civita connection $ \nabla $ is metric and torsion free and all $ G $-fundamental fields commute with $ \xi_i $ ($ G $ acts by 3-Sasakian automorphisms), we have
\[ (\xi_i)_p\big(g(\overline{Y},\overline{Z})\big) = g_p(\nabla_{(\xi_i)_p} \overline{Y}, \overline{Z}_p) + g_p(\overline{Y}_p, \nabla_{(\xi_i)_p} \overline{Z}) = g_p(\nabla_{\overline{Y}_p} \xi_i,\overline{Z}_p) + g_p(\overline{Y}_p, \nabla_{\overline{Z}_p}\xi_i) \, . \]
Finally, $ \nabla \xi_i = - \varphi_i $ and $ g(\cdot, \varphi_i \cdot) $ is skew-symmetric.
\end{proof}

We now move on to the complex picture and let $ \mf{u} := \mf{g}^\C $, $ \mf{v} := \mf{h}^\C $, $ \mf{c} := \mf{t}^\C $ and $ \alpha := 2i \alpha_1^\C|_\mf{c} $. Let us consider the vectors $ H_\alpha,X_\alpha,Y_\alpha \in \mf{u} $ defined by
\[ H_\alpha := \frac{1}{i} X_1 \, , \quad X_\alpha := \frac{1}{2i}(X_2 - i X_3) \, , \quad Y_\alpha := \frac{1}{2i}(X_2 +iX_3) \, , \]
which satisfy the commutation relations
\[ [H_\alpha,X_\alpha] = 2X_\alpha \, , \quad [H_\alpha,Y_\alpha] = -2Y_\alpha \, , \quad [X_\alpha,Y_\alpha] = H_\alpha \, . \]

\begin{proposition} \label{rootprop}
The linear form $ \alpha $ is a root of $ \mf{u} $ with respect to $ \mf{c} $, whose root space is given by $ \mf{u}_\alpha = \langle X_\alpha \rangle $. Furthermore, $ \mf{u}_{-\alpha} = \langle Y_\alpha \rangle $ and $ H_\alpha $ is the coroot of $ \alpha $.
\end{proposition}
\begin{proof}
Firstly, $ [H_\alpha,X_\alpha] = 2X_\alpha = \alpha(H_\alpha) X_\alpha $. Since $ X_2, X_3 $ commute with $ \mf{h} $, the vector $ X_\alpha $ commutes with $ \mf{v} $ and in particular with $ \mf{s}^\C $. Likewise, $ \mf{\alpha}_1 $ vanishes on $ \mf{h} $, so that $ \alpha $ vanishes on $ \mf{v} $ and in particular on $ \mf{s}^\C $.
\end{proof}

Let $ \Phi \subset \mf{c}^*$ denote the root system of $ \mf{u} $ with respect to $ \mf{c} $. We consider the $ \Z $-grading of $ \mf{u} $ introduced in \Cref{roots}, viz.
\[ \mf{u}^{(k)} := \bigoplus_{\substack{\beta \in \mf{c}^* \, , \\ c_{\alpha\beta} = k}} \mf{u}_\beta \, . \]

\begin{lemma} \label{0comp}
The $ 0 $-component of the grading is given by $ \mf{u}^{(0)} = \mf{v} \oplus \langle H_\alpha \rangle $.
\end{lemma}
\begin{proof} $ \mf{u}^{(0)} = \ker \ad_{H_\alpha} = C_\mf{u}(H_\alpha) = \mf{v} \oplus \langle H_\alpha \rangle $.
\end{proof}

\begin{lemma} \label{pm2comp}
The $ \pm 2 $-component of the grading is given by $ \mf{u}^{(\pm 2)} = \mf{u}_{\pm \alpha} $.
\end{lemma}
\begin{proof} Suppose there was a root $ \beta \neq \alpha $ such that $ c_{\alpha\beta} = 2 $. Then $ \langle \beta, \alpha \rangle > 0 $ and $ \beta - \alpha $ was a root satisfying $ c_{\alpha(\beta - \alpha)} = 0 $. We would need to have $ [\mf{u}_{\alpha},\mf{u}_{\beta - \alpha}] = \mf{u}_\beta $, but the previous lemma implies $ \mf{u}_{\beta-\alpha} \subset \mf{u}^{(0)} = \mf{v} \oplus \langle H_\alpha \rangle $ and $ [\mf{u}_\alpha,\mf{v}] = 0 $, $ [\mf{u}_\alpha, H_\alpha] = \mf{u}_\alpha $.
\end{proof}
\begin{proposition}\label{gusimple}
The Lie algebras $ \mf{g} $ and $ \mf{u} $ are simple.
\end{proposition}
\begin{proof}
The semisimple Lie algebra $ \mf{g} $ decomposes as a direct sum $ \mf{g} = \mf{g}_1 \oplus \ldots \oplus \mf{g}_m $ of simple ideals. Since the Killing form of $ \mf{g} $ is negative definite, the same applies to the ideals $ \mf{g}_i $, which thus cannot be the realification of a complex Lie algebra. Therefore, their complexifications $ \mf{u}_i := \mf{g}_i^\C $ are also simple and yield a decomposition $ \mf{u} = \mf{u}_1 \oplus \ldots \oplus \mf{u}_m $ into simple ideals \cite[Theorem 6.94]{Knap}. Accordingly, the root system is a disjoint union $ \Phi = \Phi_1 \sqcup \ldots \sqcup \Phi_m $. We claim that $ \mf{g} = \mf{g}_i $ (and hence $ \mf{u} = \mf{u}_i $), where $ i $ is the unique index such that $ \alpha \in \Phi_i $. \\
For $ j \neq i $, the ideal $ \mf{g}_j $ commutes with $ \mf{g}_i \supset (\mf{u}_\alpha \oplus \mf{u}_{-\alpha}) \cap \mf{g} \ni X_2,X_3 $, so $ \mf{g}_j \subset \mf{h} = \mf{g}_p $. Since $ \mf{g}_j $ is an ideal and $ G $ is connected, it follows that $ \mf{g}_j = \Ad_g(\mf{g}_j) \subset \Ad_g(\mf{g}_p) = \mf{g}_{g \cdot p} $ for all $ g \in G $. Because the $ G $-action is almost effective, we must have $ \mf{g}_j = 0 $.
\end{proof}

It is well-known that for any root system $ \Phi $ and any roots $ \alpha,\beta \in \Phi $, the Cartan numbers are bounded by $ |c_{\alpha\beta}| \leq 3 $. Furthermore, the only \emph{irreducible} case where $ |c_{\alpha\beta}| = 3 $ occurs is when $ \mf{g} = \mf{g}_2 $, $ \alpha $ is one of the short roots and $ \beta $ is the long root that forms an angle of 150 (210) degrees with $ \alpha $. We relegate the proof that this case cannot actually occur in our situation to the next section. \\
In all the remaining cases, we have therefore shown that $ \alpha $ is a maximal root (cf. \autoref{maxroot}), so we may carry out the construction from \Cref{const}. We now prove that the 3-Sasakian structure obtained this way indeed coincides with the original one we started with. We simplify the analysis by studying the reductive complement $ \mf{m} := \mf{h}^{\perp_B} $.

\begin{lemma} \label{decomp} The reductive complement $ \mf{m} $ decomposes $ B $-orthogonally as
\[ \mf{m} = \langle X_1,X_2,X_3 \rangle \oplus \bigoplus_{\substack{\beta \in \Phi \, , \\ c_{\alpha\beta} = 1}} (\mf{u}_\beta \oplus \mf{u}_{-\beta}) \cap \mf{g} =: \mf{k} \oplus \mf{g}_1 \, . \]
\end{lemma}
\begin{proof}
$ \supset $: Clearly, $ X_1,X_2,X_3 $ are $ B $-orthogonal to $ \mf{h} $. For $ \beta, \gamma \in \mf{c}^* $ with $ \beta + \gamma \neq 0 $, the subspaces $ \mf{u}_\beta $ and $ \mf{u}_\gamma $ are $ B^\C $-orthogonal. This implies that for all $ \beta \in \Phi $ with $ c_{\alpha\beta} =1 $, the subspaces $ \mf{u}_{\pm\beta} $ are also $ B^\C $-orthogonal to $ \mf{h} $. \\
$ \subset $: By Lemmata \ref{0comp} and \ref{pm2comp}, both sides of the equation have dimension $ 4n+3 $.
\end{proof}

We can compare the structure tensors of the two 3-Sasakian structures in question via the isomorphism $ \psi :\mf{m} \to T_pM , \, X \mapsto \overline{X}_p $. \autoref{reeb} has already shown that the vectors $ X_i $ correspond to the Reeb vector fields $ \xi_i $. Looking back at \autoref{goal}, equality of the contact forms is equivalent to the following

\begin{lemma}
$ \alpha_i = -B(X_i, \cdot)/4(n+2) $.
\end{lemma}
\begin{proof}
By definition, $ \alpha_i = B(X_i, \cdot) /B(X_i,X_i) $. We have
\begin{align*}
B(X_1,X_1) &= B^\C(iH_\alpha,iH_\alpha) = - B^\C(H_\alpha,H_\alpha) = -\tr \ad_{H_\alpha}^2 \\
&= - 4 \cdot (\dim \mf{u}^{(2)} + \dim \mf{u}^{(-2)}) - 1 \cdot (\dim \mf{u}^{(1)} + \dim \mf{u}^{(-1)}) = -4(n+2) \, . 
\end{align*}
We also have $ B(X_2,X_2) = B(X_3,X_3) = - 4(n+2) $, since we could have used the same arguments for a maximal torus of e.g.~the form $ \mf{s} \oplus \langle X_2\rangle $.
\end{proof}

Because the contact forms coincide, so do their differentials, which are the fundamental 2-forms. Since the Riemannian metrics are determined by the fundamental 2-forms together with the almost complex structures, it suffices to show that the latter coincide. Let $ L_i : \mf{m} \to \mf{m} $ denote the $ \Ad(H) $-invariant endomorphism of $ \mf{m} $ corresponding to the $ G $-invariant endomorphism field $ \varphi_i $, i.e.~$ L_i = \psi^{-1} \circ (\varphi_i)_p \circ \psi $. Looking back at \autoref{goal}, the claim reduces to showing that
\[ L_i|_\mf{k} = \frac{1}{2} \ad_{X_i} \, , \qquad L_i|_{\mf{g}_1} = \ad_{X_i} \, . \]
The first equation is clear from \autoref{reeb} and the 3-Sasaki equations in \autoref{3Sdef}. The following lemma shows that $ L_1 $ is not only $ \ad(\mf{h}) $-invariant, but even $ \ad(\mf{t}) $-invariant:

\begin{lemma} \label{commute}
The endomorphisms $ L_1 $ and $ \ad_{X_1} $ commute on $ \mf{g}_1 $.
\end{lemma}
\begin{proof}
For all $ Y, Z \in \mf{g}_1 $, we have
\begin{align*}
2g_p(\overline{Y}_p, \overline{L_1Z}_p) &= d\eta_1(\overline{Y}_p,\overline{Z}_p) \\
&= d\eta_1(\overline{[X_1,Y]}_p, \overline{[X_1,Z]}_p) \\
&= 2g_p(\overline{[X_1,Y]}_p, \overline{L_1[X_1,Z]}_p) \\
&= - 2 g_p(\overline{Y}_p, \overline{[X_1,L_1[X_1,Z]}_p) \, .
\end{align*}
In the second equation, we used that $ \ad_{X_1} $ corresponds to an almost complex structure on $ \mf{g}_1 $ which is compatible with the common fundamental 2-form $ d\eta_1 $. The last equation follows from \autoref{adtinv}. This shows that $ L_1 = - \ad_{X_1} \circ L_1 \circ \ad_{X_1} $ on $ \mf{g}_1 $ and consequently, $ \ad_{X_1} \circ L_1 = - \ad_{X_1}^2 \circ L_1 \circ \ad_{X_1} = L_1 \circ \ad_{X_1} $.
\end{proof}

\begin{proposition}
The almost complex structures of the two 3-Sasakian structures in question coincide.
\end{proposition}
\begin{proof}
Let $ \beta $ be a root such that $ c_{\alpha\beta} = 1 $. Since $ \ad_{H_\alpha} $ leaves $ \mf{u}_\beta $ invariant, so does $ \ad_{X_1}^\C $. Because $ L_1 $ is $ \ad(\mf{t}) $-invariant by \autoref{commute}, $ L_1^\C $ is $ \ad(\mf{c}) $-invariant and thus also leaves $ \mf{u}_\beta $ invariant. Now, $ \ad_{X_1}^\C $ and $ L_1^\C $ are $ \C $-linear maps on the one-dimensional subspace $ \mf{u}_\beta $ which square to $ -\id $, so they must be given by multiplication with $ \pm i $. Since both endomorphisms commute with complex conjugation, they act on $ \mf{u}_{-\beta} = \overline{\mf{u}_\beta} $ by multiplication with $ \mp i $. Therefore, $ L_1 $ and $ \ad_{X_1} $ coincide on $ (\mf{u}_\beta \oplus \mf{u}_{-\beta}) \cap \mf{g} $ up to sign. We finish the proof that $ L_1 = \ad_{X_1} $ on $ \mf{g}_1  $ by observing that for $ Y \in \mf{g}_1, Y \neq 0 $, \autoref{lieder} implies
\begin{align*}
2g_p(\overline{[X_1,Y]}_p, \overline{L_1Y}_p) &= d\eta_1(\overline{[X_1,Y]}_p,\overline{Y}_p) \\
&= \eta_1([\overline{[X_1,Y]},\overline{Y}]_p) = - \alpha_1([[X_1,Y],Y]) \\
&= \frac{B(X_1,[[X_1,Y],Y])}{4(n+2)} = - \frac{B([X_1,Y],[X_1,Y])}{4(n+2)} > 0 \, .
\end{align*}
Again, we could have repeated the arguments for a maximal torus of e.g.~the form $ \mf{s} \oplus \langle X_2 \rangle $. Even though the root spaces would have looked differently then, the subalgebra $ \mf{g}_1 $ would have still been the same because it can be defined independently of the maximal torus as the $ B $-orthogonal complement of $ \mf{k} $ in $ \mf{m} $ by virtue of \autoref{decomp}.
\end{proof}

\begin{remark}
In later sections, instead of working with the simply connected, almost effectively acting Lie group $ G $ with Lie algebra $ \mf{g} $, we may sometimes turn to a non-simply connected (possibly effectively acting) group $ \widetilde{G} $ with Lie algebra $ \mf{g} $. For $ \mf{g} = \mf{so}(k) $, using $ \widetilde{G} = SO(k) $ instead of $ G = Spin(k) $ allows us to describe the corresponding coset space more explicitly via matrices. If we consider a description $ \widetilde{G}/\widetilde{H} $, then the isotropy group of the $ G $-action on $ \widetilde{G}/\widetilde{H} $ is given by the connected subgroup $ H \subset G $ whose Lie algebra coincides with that of $ \widetilde{H} $. This follows from the fact that $ M $ is simply connected via the long exact sequence of homotopy groups. Hence, $ \widetilde{G}/\widetilde{H} $ and $ G/H $ are governed by the same Lie algebraic data and are therefore isomorphic homogeneous 3-Sasakian manifolds.
\end{remark}

\section{Why the Short Root of $ \mf{g}_2 $ Cannot Occur} \label{G2}

We need to fill the final gap left in the proof of \autoref{conv} in the previous section:

\begin{proposition} \label{g2}
Even in the case of a homogeneous 3-Sasakian manifold with automorphism algebra $\mf{g}_2 $, the root described in \Cref{univ} is maximal.
\end{proposition}

For the sake of contradiction, let us assume that $ \alpha $ was one of the short roots of $ \mf{g}_2 $. Again, we consider the reductive complement $ \mf{m} := \mf{h}^{\perp_B} $ as well as the maps $ \psi: \mf{m}\to T_pM $, $ X \mapsto \overline{X}_p $ and $ L_i := \psi^{-1} \circ (\varphi_i)_p \circ \psi: \mf{m} \to \mf{m} $. Using the same arguments as in the proof of \autoref{decomp}, we obtain the $ B $-orthogonal decomposition
\[ \mf{m} = \langle X_1,X_2,X_3 \rangle \oplus \bigoplus_{\substack{\beta \in \Phi \, , \\ c_{\alpha\beta} \in \{1,3\}}} (\mf{u}_\beta \oplus \mf{u}_{-\beta}) \cap \mf{g} \, . \]
Under the isomorphism $ \psi: \mf{m} \to T_pM $, this induces a decomposition \mbox{of the tangent space:}
\[ T_pM = \langle \xi_1, \xi_2,\xi_3 \rangle \oplus \bigoplus_{\substack{\beta \in \Phi \, , \\ c_{\alpha\beta} \in \{1,3\}}} V_\beta \, , \]
where $ V_\beta := \psi((\mf{u}_\beta \oplus \mf{u}_{-\beta}) \cap \mf{g}) $.

\begin{lemma} \label{gorthog}
The above decomposition of $ T_pM $ is $ g_p $-orthogonal.
\end{lemma}
\begin{proof}
If $ Y  \in (\mf{u}_\beta \oplus \mf{u}_{-\beta}) \cap \mf{g} $, then
\[ g_p((\xi_i)_p, \overline{Y}_p) = \eta_i(\overline{Y}_p)  = \alpha_i(Y) = B(X_i,Y)/B(X_i,X_i) = 0 \, , \]
Hence, each $ V_\beta $ is $ g_p $-orthogonal to $ \langle \xi_1,\xi_2,\xi_3\rangle $. If $ \beta_1, \beta_2 $ are roots such that $ \beta_1 \neq -\beta_2 $, then there exists some $ X \in \mf{c} $ such that $ \beta_1(X) \neq -\beta_2(X) $. We extend $ \psi $ and $ g_p $ complex (bi-)linearly, let $ Y \in \mf{u}_{\beta_1} $, $ Z \in \mf{u}_{\beta_2} $ and complexify \autoref{adtinv} to obtain
\[ \beta_1(X) g_p(\psi Y, \psi Z) = g_p(\psi[X,Y], \psi Z) = -g_p(\psi Y, \psi[X,Z]) = -\beta_2(X) g_p(\psi Y, \psi Z) \, . \]
Since $ \beta_1(X) \neq -\beta_2(X) $, it follows that $ \psi \mf{u}_{\beta_1} $ and $ \psi \mf{u}_{\beta_2} $ are $ g_p $-orthogonal. This implies that for $ \beta \neq \pm \gamma $, the subspaces $ V_\beta $ and $ V_\gamma $ are $ g_p $-orthogonal.
\end{proof}

\begin{lemma} \label{gBorthog}
For all $ Y, Z \in \mf{g} $, we have
\[ g_p(\overline{Y}_p, \overline{L_iZ}_p) = 0 \iff B(X_i,[Y,Z]) = 0 \, . \]
\end{lemma}
\begin{proof} By virtue of \autoref{lieder},
\[ 2g_p(\overline{Y}_p,\overline{L_iZ}_p) = d\eta_i(\overline{Y}_p, \overline{Z}_p) = \eta_i([\overline{Y}, \overline{Z}]_p) = - \alpha_i([Y,Z]) = - \frac{B(X_i,[Y,Z])}{B(X_i,X_i)} \, .\qedhere \]
\end{proof}

\begin{lemma} \label{phi2}
For any root $ \beta \in \Phi $, we have $ \varphi_2 V_\beta \subset V_{\beta+\alpha} \oplus V_{\beta-\alpha} $.
\end{lemma}
\begin{proof}
Let $ \gamma \in \Phi $ such that $ \gamma \neq \sigma \alpha + \tau \beta $ for all $ \sigma, \tau \in \{\pm 1\} $. Then, $ \sigma\beta+\tau\gamma \not\in \{\pm \alpha\} $ for all $ \sigma, \tau \in \{\pm1\} $. Consequently, the subspace $ [\mf{u}_\beta \oplus \mf{u}_{-\beta}, \mf{u}_\gamma \oplus \mf{u}_{-\gamma}] $ is $ B^\C $-orthogonal to $ \mf{u}_\alpha  \oplus \mf{u}_{-\alpha} \ni X_\alpha, Y_\alpha $ and thus also to $ X_2  = i(X_\alpha +Y_\alpha) $ (see the equations above \autoref{rootprop}). The previous lemma now implies that $ \varphi_2 V_\beta $ is $ g_p $-orthogonal to $ V_\gamma $. The claim then follows from \autoref{gorthog}.
\end{proof}

\begin{proof}[Proof of \autoref{g2}]
Let us label some of the roots of $ \mf{g}_2 $ according to the following diagram:

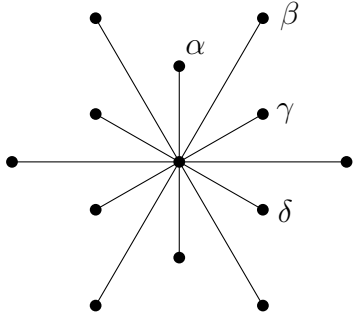
\begin{figure}[H]
\begin{center}
\begin{tikzpicture}[scale=2.2]
\draw (0:1cm)--(0,0)--(180:1cm);
\draw (90:0.577cm)--(0,0)-- (270:0.577cm);
\draw (30:0.577cm)--(0,0)--(210:0.577cm);
\draw (60:1cm)--(0,0)--(240:1cm);
\draw (120:1cm)--(0,0)--(300:1cm);
\draw (150:0.577cm)--(0,0)--(330:0.577cm);

\fill  (0,0) circle (1pt);
\fill  (0:1cm) circle (1pt);
\fill  (60:1cm) circle (1pt);
\fill  (120:1cm) circle (1pt);
\fill  (180:1cm) circle (1pt);
\fill  (240:1cm) circle (1pt);
\fill  (300:1cm) circle (1pt);
\fill  (30:0.577cm) circle (1pt);
\fill  (90:0.577cm) circle (1pt);
\fill  (150:0.577cm) circle (1pt);
\fill  (210:0.577cm) circle (1pt);
\fill  (270:0.577cm) circle (1pt);
\fill  (330:0.577cm) circle (1pt);

\node at (82:0.7cm) {\large $\alpha$};
\node at (53:1.1cm) {\large $\beta$};
\node at (25:0.7cm) {\large $\gamma$};
\node at (335:0.7cm) {\large $\delta$};
\end{tikzpicture}
\end{center}
\vfill
\end{figure}

\autoref{phi2} implies that $ \varphi_2V_\beta \subset V_\gamma $. Since $ \varphi_2 $ is injective on the horizontal space, we in fact have $ \varphi_2V_\beta = V_\gamma $. Another application of \autoref{phi2} yields $ \varphi_2 V_\gamma \subset V_\beta \oplus V_\delta $. If we can show that there exists some $ Y \in V_\gamma $ such that $ \varphi_2Y $ has a non-trivial $ V_\delta $-component, then we arrive at a contradiction to the fact that $ \varphi_2^2 = -\id $ on $ V_\beta $. \\
Let $ Z^* $ denote the complex conjugate of a vector $ Z \in \mf{u} $. We can choose $ X_\gamma \in \mf{u}_\gamma $, $ X_\delta \in \mf{u}_\delta $ in such a way that
\[ [X_\gamma, X_\delta^*] = X_\alpha =  \frac{1}{2i} (X_2 - iX_3) \, . \]
We note that
\[ X_\alpha^* = - \frac{1}{2i}(X_2+iX_3) = -Y_\alpha \, , \quad i(X_\alpha - X_\alpha^*) = X_2 \, , \]
and
\[ B([X_\gamma + X_\gamma^*, i(X_\delta^* - X_\delta)],X_2) = B(i(X_\alpha - X_\alpha^*), X_2) = B(X_2,X_2) \neq 0 \, . \]
\autoref{gBorthog} finally implies that $ \varphi_2Y $ has a $ V_\delta $-component for $ Y := (\overline{X_\gamma + X_\gamma^*})_p $.
\end{proof}

\section{Why no Proper Subgroup of $ \Aut_0(M) $ Acts Transitively}\label{AutM}

We have shown that any simply connected homogeneous $3$-Sasakian manifold is obtained from a complex $3$-Sasakian datum as explained in \Cref{const}. It remains to be shown that there is a unique such description. Therefore, let us now assume that $ \widetilde{G} $ is a connected Lie group with Lie algebra $ \mf{g} $ that acts effectively and transitively on $M$, so that $\widetilde{G}\subset \mathrm{Aut}_0(M)$. We show that then $\widetilde{G}=\mathrm{Aut}_0(M)$. In other words:

\begin{proposition}\label{GAut}
No proper subgroup of $\mathrm{Aut}_0(M)$ can act transitively on a simply connected homogeneous 3-Sasakian manifold $M$.
\end{proposition}

This will conclude the proof of \autoref{goal}.\\
We consider a simply-connected homogeneous $3$-Sasakian manifold $M=\widetilde{G}/\widetilde{H}$, with notation as in \autoref{goal}: In particular, we have a reductive decomposition $\mf{g} = \mf{h} \oplus \mf{m}$, where $\mf{m} = \mf{h}^\perp_B$, and $\mf{m} = \mf{k} \oplus \mf{g}_1$, where $\mf{k}$ corresponds to the Reeb vector fields $\xi_i$ and $\mf{g}_1= (\mf{h}\oplus \mf{k})^\perp_B$.  Note that we have the commutator relations (cf.~\autoref{rootconst})
\begin{align}\label{comm}\tag{$ \ast $}
[\mf{h},\mf{h}]\subset \mf{h},\, [\mf{h},\mf{k}]=0,\, [\mf{h},\mf{g}_1]\subset \mf{g}_1,\, [\mf{k},\mf{k}]\subset \mf{k},\, [\mf{k},\mf{g}_1]\subset \mf{g}_1,[\mf{g}_1,\mf{g}_1]\subset\mf{g}_0= \mf{h}\oplus \mf{k}.
\end{align}

Let $\nabla$ be the Levi-Civita connection on $M$, and $\alpha:\mf{m}\times \mf{m}\to \mf{m}$ the associated Nomizu operator defined by
\[
\alpha(X,Y) = \nabla_{X_{eH}} Y - [X,Y]_{eH}.
\]
It satisfies
\begin{align*}
\alpha(X,Y) = \begin{cases} 0 & X\in \mf{k}\text{ and } Y\in \mf{g}_1 \\ \frac12 [X,Y]_{\mf{m}} & X,Y\in \mf{k} \text{ or } X,Y\in \mf{g}_1 \\
[X,Y]_{\mf{m}} & X\in \mf{g}_1 \text{ and } Y\in \mf{k}, \end{cases}
\end{align*}
see \cite[Theorem 4.2]{DOP}. Thus, by definition of the Nomizu operator, we have
\[
\nabla_{X_{eH}}Y = L_Y(X),
\]
where $L_Y:\mf{m}\to \mf{m}\cong T_{eH}M$ is defined by
\begin{align*}
L_Y(X) = \begin{cases} \frac32 [X,Y]=-3\sum\eta_i(Y_{eH})\varphi_i X_{eH} & X\in \mf{k} \\  2[X,Y]=-2\sum\eta_i(Y_{eH})\varphi_iX_{eH} & X\in \mf{g}_1 \end{cases}
\end{align*}
for $Y\in \mf{k}$ and
\begin{align*}
L_Y(X) = \begin{cases} [X_i,Y] = \varphi_i(Y_{eH}) & X=X_i\in \mf{k} \\ \frac32 [X,Y]_{\mf{k}} = -\frac32 \sum_{i=1}^3 d\eta_i(X_{eH},Y_{eH})\xi_i& X\in \mf{g}_1 \end{cases}
\end{align*}
for $Y\in \mf{g}_1$, where we used Lemma \ref{lieder} in the last equation. Hence every vector from $\mf{m}=\mf{k}\oplus\mf{g}_1$ satisfies
\begin{align}\label{nablaY}
\nabla_{X_{eH}}Y&=-3\sum_{i,j=1}^3\eta_i(X_{eH})\eta_j(Y_{eH})\xi_k-2\sum_{j=1}^3\eta_j(Y_{eH})\varphi_j(X_{eH})_\mathcal{H}\tag{$\#$}\\
&\quad +\sum_{i=1}^3\eta_i(X_{eH})\varphi_i(Y_{eH})_\mathcal{H}-\frac 32 \sum_{i=1}^3d\eta_i((X_{eH})_{\mathcal{H}},(Y_{eH})_\mathcal{H})\xi_i\notag
\end{align}
for all $X\in T_{eH}M$ and where $X_\mathcal{H}$ denotes the projection to $\mathcal{H}=\bigcap \ker \eta_i$.

Consider the maps
\begin{align}\label{masKilling} \notag
\mf{m}\to \left\{Y\text{ Killing field on }M\mid Y \text{ satisfies \eqref{nablaY} for all } X\in T_{eH}M \right\}\to T_{eH}M,
\end{align}
where the first map is inclusion and the second is evaluation at $eH$. The evaluation map is injective, as for Killing fields $Y_1,Y_2$ in the middle space with $(Y_1)_{eH} = (Y_2)_{eH}$, also $(\nabla Y_1)_{eH} = (\nabla Y_2)_{eH}$, which implies that $Y_1 = Y_2$. Since $\mf{m}\cong T_{eH}M$ both maps are isomorphisms. In particular, we have shown that \begin{align*}
\mf{m}=\{Y\text{ Killing field on }M\mid Y \text{ satisfies \eqref{nablaY} for all } X\in T_{eH}M \}\subset\mf{isom}(M).
\end{align*}
Consider the subspace $[\mf{m},\mf{m}]+ \mf{m}\subset \mf{g}$. Using \eqref{comm} we find that this is an ideal in $\mf{g}$ and thus (Proposition \ref{gusimple}) already $\mf{g}$ itself. Hence, the knowledge of $\mf{m}$ as a subset of the Lie algebra of Killing vector fields on $M$ alone determines $\mf{g}$.

\begin{proof}[Proof of \autoref{GAut}]
By the above argument, every connected Lie group $ \widetilde{G} $ with Lie algebra $ \mf{g} $ acting effectively and transitively on $ M $ has the same Lie algebra, namely $ \mf{g} = \mf{aut}(M) $. The corresponding connected subgroup of $ \Aut(M) $ is therefore $ \widetilde{G} = \Aut_0(M) $.
\end{proof}

\section{Determining the Isotropy}\label{Isotropy}

Having proven \autoref{goal}, we now derive the precise list given in \autoref{3Slist}. By Theorems \ref{constthm} and \ref{conv} any simply-connected homogeneous $3$-Sasakian manifold can be written in the form $G/H$, where $G$ is a simply-connected simple Lie group and $H = (C_G(K))_0$, where $K\subset G$ is the connected subgroup with Lie algebra $\mf{k}\cong \mf{sp}(1)$ determined by a maximal root. In this section, we will determine the isotropy groups $H$, thereby proving Corollary \ref{3Slist} in the simply connected case. The classical cases are dealt with in the following
\begin{proposition}
For $ G = Sp(n+1), SU(m) $ and $ \widetilde{G} = SO(k) $, the isotropy groups are given by $ H = Sp(n), S(U(m-2) \times U(1)) $ and $ \widetilde{H} = SO(k-4) \times Sp(1) $, respectively.
\end{proposition}
\begin{proof}
We use the explicit description of the root systems of the compact groups provided in \cite[Chapter 11]{Tapp}. \\
$ G = Sp(n+1) $: We may choose the maximal root $ \alpha $ such that \mbox{$ \mf{k} = \{\diag(0_n,\mf{sp}(1) ) \} $} (by letting $ \alpha = \gamma_{n+1} $ in \textsc{Tapp}'s notation). Accordingly,
\begin{align*}
K &= \{\diag(I_n,Sp(1)) \} \, , \\
C_G(K) &= \{\diag(Sp(n),\pm 1)\} \, , \\
H &= (C_G(K))_0 = \{\diag(Sp(n),1) \} \, .
\end{align*}
$ G = SU(m) $: We may choose $ \alpha $ such that $ \mf{k} = \{\diag(0_{m-2}, \mf{su}(2))\} $ (by letting $ \alpha = \alpha_{m-1,m} $ in \textsc{Tapp}'s notation). Accordingly,
\begin{align*}
K &= \{\diag(I_{m-2},SU(2)) \} \, , \\
C_G(K) &= H = \{\diag(SU(m-2),zI_2) \; | \; z \in U(1)\} \cap SU(m) \, .
\end{align*}
$ \widetilde{G} = SO(k) $: We recall that there are two embeddings $ Sp(1)^+, \, Sp(1)^- \subset SO(4) $, depending on whether $ Sp(1) $ is viewed as acting on $ \mathbb{H} \cong \R^4 $ by multiplication from the left or right, respectively. We may choose $ \alpha $ such that $ \mf{k} = \{\diag(0_{k-4}, \mf{sp}^-(1)) \} $ (by letting $ \alpha = \alpha_{[k/2]-1,[k/2]} $ in \textsc{Tapp}'s notation). Accordingly,
\begin{align*}
\widetilde{K} &= \{\diag(I_{k-4},Sp^-(1)) \} \, , \\
C_{\widetilde{G}}(\widetilde{K}) &= \widetilde{H} = \{\diag(SO(k-4),Sp^+(1))\} \, .
\end{align*}
\end{proof}

We now present a different method based on Borel-de Siebenthal theory, which allows us to first understand the isotopy algebra $ \mf{h} $ in the exceptional cases: \\
Using the same notation as before, we let $\mf{s}\subset \mf{h}$ be a maximal Abelian subalgebra and consider the maximal Abelian subalgebra $\mf{t}:=\mf{s} \oplus \langle X_1\rangle$ of $\mf{g}$. Let $\alpha$ denote the maximal root that vanishes on $\mf{s}^\C$. We fix a set of positive roots of $\mf{g}$ using a slight perturbation of the hyperplane perpendicular to the maximal root $\alpha$. By intersecting this hyperplane with $\mf{s}$, we also obtain a notion of positive root for $\mf{h}$. By the very definition of root spaces, as $\mf{h}$ commutes with $X_1$, any root of $\mf{h}$ becomes, by extending it by $0$ on $X_1$, a root of $\mf{g}$.

\begin{proposition} The simple $\mf{h}$-roots are precisely those simple $\mf{g}$-roots perpendicular to $\alpha$.
\end{proposition} 
\begin{proof} By our notions of positivity, any $\mf{h}$-simple root is also $\mf{g}$-simple: If an $\mf{h}$-root is the sum of two positive $\mf{g}$-roots, both of them have to lie in the hyperplane perpendicular to $\alpha$. Conversely, recall that by \autoref{rootconst} the roots of $\mf{h}$ are exactly those roots $\beta$ perpendicular to the maximal root $\alpha$.
\end{proof}

We can thus determine the isotropy type of $H$ by deleting the nodes in the Dynkin diagram of $G$ corresponding to simple roots that are not perpendicular to $\alpha$. For each simple $G$, these were determined by \textsc{Borel} and \textsc{de Siebenthal} in \cite{BdS}: In the table on p.\ 219 they draw the Dynkin diagrams for every simple $\mf{g}$, extended by the lowest root (denoted $P$ in their notation). In order to find the isomorphism type of $H$ one therefore only needs to erase this lowest root, as well as all roots connected to it. As an example, consider the Dynkin diagram of $E_6$: \\

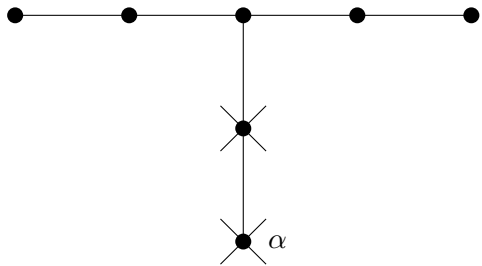
\begin{figure}[H]
\begin{center}
\begin{tikzpicture}[scale=1.5]

\fill  (0,0) circle (2pt);
\fill  (-1,0) circle (2pt);
\fill  (-2,0) circle (2pt);
\fill  (1,0) circle (2pt);

\fill  (2,0) circle (2pt);
\fill  (0,-1) circle (2pt);
\fill  (0,-2) circle (2pt);

\draw (-2,0)--(2,0);
\draw (0,0)--(0,-2);
\node at (0.3,-2) {$\alpha$};

\draw (-.2,-1.2)-- (.2,-.8);
\draw (-.2,-.8)-- (.2,-1.2);

\draw (-.2,-2.2)-- (.2,-1.8);
\draw (-.2,-1.8)-- (.2,-2.2);

\end{tikzpicture}
\end{center}
\vfill
\end{figure}
Deleting $\alpha$ as well as the unique simple root connected to it results in the Dynkin diagram of $SU(6)$: The homogeneous $3$-Sasakian manifold corresponding to $E_6$ is $E_6/SU(6)$. 
\begin{remark} 
If one removes only the nodes in the Dynkin diagram of $G$ that are connected to $\alpha$, and not $\alpha$ itself, one obtains the Dynkin diagram of the normalizer $N_G(K)$, which then yield the Wolf spaces $G/N_G(K)$. Note that by the list in \cite{BdS}, in all cases except $G=SU(n)$ the maximal root $\alpha$ is connected to only one other node, which means that in these cases the groups $H$ and $N_G(K)$ are semisimple, whereas in the case $G=SU(n)$ the groups $H$ and $N_G(K)$ have a one-dimensional center. Furthermore, in the cases except $SU(n)$, the normalizer $N_G(K)$ is a maximal subgroup of maximal rank: the types of such groups are exactly those that were classified by \textsc{Borel} and \textsc{de Siebenthal} in \cite{BdS}: Given a simple compact Lie group $G$, one adds the lowest root to the Dynkin diagram and removes one other simple root from it. 
\end{remark}

Going through the list in \cite{BdS}, one obtains the Lie algebras of the isotropy groups of the homogeneous spaces occurring in \autoref{3Slist}. As we determined the isotropy groups in the classical cases above, in order to finish the proof of this corollary in the simply connected case, we only need to argue that in the exceptional cases the isotropy groups are simply connected. \textsc{Ishitoya} and \textsc{Toda} showed in \cite[Corollary 2.2]{IT} that in the cases $G=G_2, F_4, E_6, E_7, E_8$, we have $\pi_2(G/N_G(K)) = \Z_2$, which is, because $G$ is simply connected, equivalent to $\pi_1(N_G(K))=\Z_2$. (See also \cite[Remarque II, p.\ 220]{BdS} for how to compute the fundamental group of a maximal subgroup of $G$ of maximal rank.) Moreover, by \cite[Theorem  2.1]{IT} the normalizer $N_G(K)$ is of the form $N_G(K) = (H\times Sp(1))/\Z_2$, which then implies that $H$ is simply connected.

\section{Why $ \R P^{4n+3} $ is the Only Non-Simply Connected Homogeneous 3-Sasakian Manifold} \label{fund}

Having treated the simply connected case of Corollary \ref{3Slist}, our goal is now to prove the following

\begin{theorem} \label{fundthm}
The only homogeneous 3-Sasakian manifolds which are not simply connected are the real projective spaces $ \R P^{4n+3} $.
\end{theorem}

Let $ M = G/\overline{H} $ be a homogeneous 3-Sasakian manifold (not necessarily simply connected), where $ G $ is a simply connected compact Lie group and $ \overline{H} $ is possibly disconnected. The universal cover of $ G/\overline{H} $ is given by $ G/H $, where $ H $ denotes the identity component of $ \overline{H} $, and the homogeneous 3-Sasakian structure lifts to the simply connected space $ G/H $. As shown in \Cref{univ}, the automorphism group $ G $ has to be simple. The vectors $ X_i \in \mf{g} $ from \Cref{univ} span a subalgebra $ \mf{k} := \langle X_1,X_2,X_3\rangle \cong \mf{sp}(1) $ and we let $ K \subset G $ denote the corresponding connected subgroup. \\
Since $ H $ is the identity component of $ \overline{H}$, we have $ \overline{H} \subset N_G(H) $. Furthermore, the 3-Sasakian structure descends from $ G/H $ to $ G/\overline{H} $, so $ \overline{H} \subset C_G(K) $. Conversely, any subgroup $ \overline{H} \subset N_G(H) \cap C_G(K) $ containing $ H $ allows us to define a 3-Sasakian structure on $ G/\overline{H}$. Summarizing, the non-simply connected quotients of a given simply connected homogeneous 3-Sasakian manifold $ G/H $ are classified by the subgroups of the group
\[ \big(N_G(H) \cap C_G(K)\big)/H \, . \]
Thus, it suffices to show that this quotient is $ \Z_2 $ for $ G = Sp(n+1) $ and trivial otherwise.

\begin{lemma}
The numerator $ N_G(H) \cap C_G(K) $ is the subgroup generated by $ H \cup Z(K) $.
\end{lemma}
\begin{proof}
Clearly, $ H \cup Z(K) \subset N_G(H) \cap C_G(K) $. The vector $ X_1 \in \mf{g} $ is the infinitesimal generator of a circle subgroup $ S_1 \subset K $. Since $ C_\mf{g}(X_1) = \mf{h} \oplus \langle X_1 \rangle $ and the centralizer of any torus (not necessarily maximal) in a compact connected Lie group is always connected, it follows that $ C_G(S_1) $ is the subgroup generated by $ H \cup S_1 $. Consequently, any $ g \in C_G(K) \subset C_G(S_1) $ can be represented as $ g = hg_1 $ for some $ h \in H $, $ g_1 \in S_1 $. Since $ H \subset C_G(K) $, we have $ g_1 = h^{-1}g \in C_G(K) \cap K = Z(K) $.
\end{proof}

\begin{proposition} \label{prop:quotienttrivial}
The quotient $ (N_G(H) \cap C_G(K))/H $ is $ \Z_2 $ for $ G = Sp(n+1) $ and trivial otherwise.
\end{proposition}
\begin{proof}
By the previous lemma, it suffices to check if $ Z(K) $ is contained in $ H $. \\
$ G = Sp(n+1) $: We have already seen in \Cref{Isotropy} that the center $ Z(K) = \{\diag(I_n,\pm 1)\} $ is \textbf{not} contained in $ H = \{\diag(Sp(n),1)\} $. \\
$ G = SU(m) $: We have also shown that in this case $ Z(K) = \{\diag(I_{m-2},\pm I_2)\} $ is contained in $ H = S(U(m-2) \times U(1)) $. \\
$ G = Spin(k) $: We have seen that for $ \widetilde{G} = SO(k) $, the center $ Z(K) = \{\diag(I_{k-4},\pm I_4)\} $ is contained in $ \widetilde{H} = SO(k-4) \times Sp(1)^+  $. We now transfer this statement to $ G= Spin(k)$: Denote the universal covering map by $\pi:Spin(k)\to SO(k)$. First, we observe that the connected subgroup of $Spin(r)$ that maps onto a block-diagonally embedded $SO(r-1)$ is $Spin(r-1)$, for $r\geq 4$: this is because $S^{r-1} = SO(r)/SO(r-1)$ is $2$-connected for $r\geq 4$, hence equal to $Spin(r)/Spin(r-1)$ by the long exact sequence in homotopy. Thus, for $k\geq 7$, the subgroups $SO(k-4)$ and $SO(4)$ lift to $Spin(k-4)$ and $Spin(4)$, respectively. As $\pi$ is a $2:1$-covering, the group covering $SO(k-4)\times SO(4)$ is $Spin(k-4)\times_{\Z_2} Spin(4)$, where the $\Z_2$ quotient means that the nontrivial elements in the kernels of the respective projections are identified. Now, $Spin(k-4)\times_{\Z_2} Spin(4) \cong Spin(k-4)\times_{\Z_2} (Sp(1)^+\times Sp(1)^-)$. This implies that the center of $Sp(1)^-$ is contained in $H=Spin(k-4)\times Sp(1)^+$.\\
$ G = G_2, F_4, E_6, E_7, E_8 $: \textsc{Ishitoya} and \textsc{Toda} showed that the subgroup $ U $ of the corresponding qK space $ G/U $ has to be of the form $ U = (H \times Sp(1))/\Z_2 $ and that the center $ Z(Sp(1)) $ is contained in $ H $ \cite[Theorem 2.1]{IT}.
\end{proof}

\section{Deriving the Classification of Homogeneous Positive qK Manifolds} \label{wolf}

We end this article by showing that the classification of homogeneous positive qK manifolds in \Cref{Wolf}, which had originally been the stepping stone for the classification of homogeneous 3-Sasakian manifolds, can in turn be derived from our results.

Let $ B $ be a positive qK manifold. We recall that qK manifolds may be characterized by a subbundle $ Q \subset \textnormal{End} \, TB $ of the endomorphism bundle which admits a local frame satisfying the multiplication rules of the quaternions. In her 1975 article \cite{Kon}, \textsc{Konishi} showed that the $ SO(3) $-principal fibre bundle $ P \to B $ of oriented orthonormal frames of $ Q $ admits a 3-Sasakian structure. This construction is known as the \emph{Konishi bundle} over $ B $ and constitutes the inverse of the fibration introduced in \autoref{fibr}. Another natural and interesting bundle over $ B $ is the unit sphere bundle $ Z := S(Q) $ in $ Q $, known as the \emph{twistor fibration}. Its total space $ Z $ is both a complex contact manifold and a Fano variety \cite[Chapters 12, 13]{BG}. \\
A \emph{qK automorphism of $ B $} is an isometry $ \phi: B \to B $ such that conjugation with its differential $ d\phi $ leaves the bundle $ Q $ invariant. We call $ B $ a \emph{homogeneous qK manifold} if there is a Lie group $ G $ acting transitively on $ B $ by qK automorphisms. We first show the following

\begin{proposition} \label{Konishi}
The Konishi bundle over a simply connected homogeneous positive qK manifold is a homogeneous 3-Sasakian manifold.
\end{proposition}

Let $ B $ be a simply connected homogeneous positive qK manifold, so that we may write $ B = G/U $, where $ G $ is simply connected and $ U $ is connected. Because $ G $ acts on $ B $ by qK automorphisms, the $ G $-action lifts to the Konishi bundle $ P $. In particular, the istropy group $ U $ acts on the fiber $ F $ of $ P $ over the identity coset $ eU \in B $. \\
Choose and fix a frame $ p \in F $. This allows us to identify $ F $ with $ SO(3) $ via the orbit bijection $ \theta_p: SO(3) \to F , \, g \mapsto p \cdot g $. The $ SO(3) $-left action on itself by left multiplication now induces a left action on $ F $ (which depends on the choice of $ p $), viz.~$ g \cdot_p q := \theta_p(g \theta_p^{-1}(q)) $. Since the $ U $-action commutes with the $ SO(3) $-right action on $ F $, there exists a homomorphism $ \rho: U \to V $ onto a subgroup $ V \subset SO(3) $ (again, all depending on $ p $) such that $ u \cdot q = \rho(u) \cdot_p q $ for all $ q \in F $, namely $ \rho(u) := \theta_p^{-1}(u \cdot p) $. Clearly, $ d := \dim V \in \{0,1,3\} $.

\begin{lemma}
$ d \neq 0 $.
\end{lemma}
\begin{proof}
If $ d = 0 $, then the connected group $ U $ would act trivally on $ F $. Hence, we would obtain a well-defined global section $ B \to P , \, gU \mapsto g \cdot p $, meaning that the principal fiber bundle $ P $ was trivial. But the first Pontryagin class $ p_1(P) $ of $ P $ is (up to a factor) given by the class of the fundamental $ 4 $-form $ \Omega \in \Omega^4(B) $ of $ B $ and is therefore non-trivial \cite[Proposition 14.92]{Bess}.
\end{proof}

\begin{lemma}
$ d \neq 1 $.
\end{lemma}
\begin{proof}
If $ d = 1 $, then $ V $ is a connected one-dimensional subgroup of $ SO(3) $ and is thus comprised of rotations around a fixed axis $ L \subset \R^3 $. Choose a point $ x \in L \cap S^2 $. We view the frame $ p \in F $ as a linear isometry $ \R^3 \to Q $ and consider the mapping $ B \to Z  $, $ gU \mapsto (g \cdot p)(x) $. This map is well-defined because
\[ (u \cdot p)(x) = \big(p \cdot \theta_p^{-1}(u \cdot p)\big)(x) = \big(p \cdot \rho(u)\big)(x) = p \big( \rho(u)(x)\big) = p(x) \quad \forall \; u \in U \, . \]
We would thus obtain a global section of the twistor fibration, which is impossible on compact positive qK manifolds \cite[Theorem 3.8]{AMP}.
\end{proof}

\begin{proof}[Proof of \autoref{Konishi}]
From the previous two lemmata, we know that $ U $ acts transitively on $ F $ and consequently, $ G $ acts transitively on $ P $. This action preserves the 3-Sasakian structure, since the Reeb vector fields are (by construction of the Konishi bundle) the infinitesimal generators of the $ SO(3) $-action, which commutes with $ G $.
\end{proof}

\begin{proof}[Proof of \autoref{Wolf}]
By \autoref{Konishi}, the Konishi bundle $ P $ over a simply connected homogeneous positive qK manifold $ B $ is a homogeneous 3-Sasakian manifold, i.e.~one of the manifolds listed in \autoref{3Slist}. Dividing $ P $ by the action of the group $ K \subset G $ from the previous sections, we obtain the list in \autoref{Wolf}. The statement about the Riemannian metric and quaternionic structure follows from the fact that the Konishi bundle is a Riemannian fibration. \\
Let us now assume that $ \overline{B} = G/\overline{U} $ was a non-simply connected homogeneous positive qK manifold, where $ \overline{U} $ is disconnected. Then, $ \overline{B} $ is finitely covered by $ B = G/U $, where $ U $ denotes the identity component of $ \overline{U} $. The qK structure lifts from $ \overline{B} $ to $ B $ and the Konishi bundles $ P, \overline{P} $ over $ B, \overline{B} $ form a diagram
\[\begin{tikzcd}
P \arrow[dashed]{r}{} \arrow{d}{} & \overline{P} \arrow{d}{} \\
B \arrow{r}{} & \overline{B}
\end{tikzcd} \]
We obtain the existence of the dashed equivariant map $ P \to \overline{P} $, so that $ \overline{P} $ is a non-simply connected homogeneous 3-Sasakian manifold. By \autoref{fundthm}, $ \overline{P} $ can only be $ \R P^{4n+3} $, which leads to the same quotient $ \overline{B} = Sp(n+1)/(Sp(n) \times Sp(1)) $ as $ P = S^{4n+3} $.
\end{proof}

\printbibliography

\textsc{Oliver Goertsches and Leon Roschig, Philipps-Universität Marburg, Fachbereich Mathematik und Informatik, Hans-Meerwein-Straße, 35043 Marburg} \\
\textit{E-mail addresses}: \texttt{goertsch@mathematik.uni-marburg.de}, \\
\texttt{roschig@mathematik.uni-marburg.de}

\textsc{Leander Stecker, Universität Hamburg, Fachbereich Mathematik, Bundes\-straße, 20146 Hamburg} \\
\textit{E-mail address}: \texttt{leander.stecker@uni-hamburg.de}
\end{document}